\documentclass[12pt]{article}
\usepackage{amsmath,amsthm,verbatim,amssymb,amsfonts,amscd,diagrams, graphicx, mathrsfs}
\usepackage[usenames,dvipsnames]{color}
\usepackage{amstext}
\usepackage{color}

\theoremstyle{plain}
\newtheorem{theorem}{Theorem}[section]
\newtheorem{corollary}[theorem]{Corollary}
\newtheorem{lemma}[theorem]{Lemma}
\newtheorem{proposition}[theorem]{Proposition}

\theoremstyle{definition}
\newtheorem{definition}[theorem]{Definition}
\newtheorem{example1}[theorem]{Example}
\newtheorem{example}{Example}

\theoremstyle{remark}
\newtheorem{remark}[theorem]{Remark}
\newcommand{\codim}{\text{codim}}
\newarrow{ul}---->
\newarrow{Backwards}<----   

\newcommand{\onto}{\twoheadrightarrow}

\newcommand{\td}[1]{\tilde{#1}}
\newcommand{\into}{\hookrightarrow}
\newcommand{\Z}{\mathbb{Z}}
\newcommand{\Q}{\mathbb{Q}}
\newcommand{\R}{\mathbb{R}}

\newcommand{\D}{\mathbb{D}}

\newcommand{\bd}{\partial}
\newcommand{\pf}{\pitchfork}

\newcommand{\mc}[1]{\mathcal{#1}}

\newcommand{\dlim}{\varinjlim}

\newcommand{\Hom}{\text{Hom}}

\newcommand{\mf}{\mathfrak}

\newcommand{\im}{\text{im}}

\newarrow{Onto}----{>>}
\newarrow{Equals}=====

\newcommand{\Ker}{\mbox{Kernel }}
\newcommand{\pfa}{\pf}
\newcommand{\pfb}{\hat \pf}
\newcommand{\pfc}{\bar \pf}

\begin{document}
\title{Additivity and non-additivity for perverse signatures}
\author{
Greg Friedman \thanks{Partially supported by MSRI}\\ Texas Christian University
\and
Eugenie Hunsicker \thanks{Partially supported by
MSRI}
\\ Loughborough University
}
\date{\today}


\maketitle



\begin{abstract}
A well-known property of the signature of closed oriented 4n-dimensional manifolds is Novikov additivity, which states that if a manifold is split into two manifolds with boundary along an oriented smooth hypersurface, then the signature of the original manifold equals the sum of the signatures of the resulting manifolds with boundary.  Wall showed that this property is not true of signatures on manifolds with boundary and that the difference from additivity could be described as a certain Maslov triple index. Perverse signatures are signatures defined for any oriented stratified pseudomanifold, using the intersection homology groups of Goresky and MacPherson.  In the case of Witt spaces, the middle perverse signature is the same as the Witt signature.  This paper proves a generalization to perverse signatures of Wall's non-additivity theorem for signatures of manifolds with boundary.  Under certain topological conditions on the dividing hypersurface, Novikov additivity for perverse signatures  may be deduced as a corollary.  In particular, Siegel's version of Novikov additivity for Witt signatures is a special case of this corollary.
\end{abstract}
\tableofcontents

\section{Introduction}

The signature of compact $4n$-dimensional oriented manifolds is an interesting and important 
manifold invariant.  It satisfies a number of remarkable properties commonly referred to as
the `signature package'.  These include cobordism invariance \cite{Thom}, equality to the index
of the signature operator and to the $L$-genus \cite{Hirz}, and Novikov additivity \cite{AS}. 
Signature has been used to prove 
various obstruction theorems.  For instance, Rokhlin's theorem (\cite{Ro}) shows that for
a $4n$-dimensional smooth compact oriented manifold to carry a spin-structure, its signature
must be divisible by 16.  Following on the heels of the successes in topology and geometric
topology of smooth compact manifolds in the 1950s and 1960s, including this work
on the signature package, mathematicians began to explore which of the results
from the smooth compact setting might be generalized to the setting of singular spaces.
In the years since then, there have been a number of interesting developments, such as the theory of intersection homology signatures on Witt spaces. In \cite{Hun07}, the
second author of this paper defined a family of `perverse signatures',  based on the intersection homology groups of Goresky and MacPherson, that may be defined
for any oriented $4n$-dimensional closed stratified pseudomanifold, though
this signature cannot be a bordism-invariant of all oriented closed pseudomanifolds, as is evident by considering the cone on a manifold with nonzero signature.
In this paper, we identify when a generalization 
of Novikov additivity holds for these signatures, as well as identifying the additivity defect 
in the case that it does not.  In future papers, we will explore what further aspects of the signature
package hold for perverse signatures.

More details will be given below, but to explain briefly our main results, Theorem \ref{T: Wall} and Corollary
\ref{C: cor}, recall that for a closed oriented $n$-dimensional pseudomanifold $X$ and perversity parameters $\bar p, \bar q$ such that $\bar p+\bar q=\bar t$, there is a duality isomorphism of intersection homology groups\footnote{If $X$ has no codimension one strata and the perversity parameters $\bar p$ and $\bar q$ satisfy the conditions of Goresky and MacPherson \cite{GM1}, these are the intersection homology groups of Goresky and MacPherson \cite{GM1,GM2}. For more general perversities or pseudomanifolds with codimension one strata, these are the intersection homology groups with ``stratified coefficients'' of the first author; see \cite{GBF10,GBF23,GBF26}. We shall follow the practice of \cite{GBF25} and omit the symbol $G_0$ utilized previously. Note, however, that in general these groups will depend on the stratification of $X$. Furthermore, if $X$ has codimension one strata, ``closed'' here really means ``s-closed'' as defined below in Section \ref{S: Review}. We omit that notation here for the sake of simplicity in the Introduction.} with rational\footnote{Throughout the paper, all results stated for $\Q$ would also hold for coefficients in $\R$.} coefficients,

 $$I^{\bar p}H_i(X;\Q)\cong \Hom(I^{\bar q}H_{n-i}(X;\Q),\Q),$$ determined by the intersection pairing $$\pfb: I^{\bar p}H_i(X;\Q)\otimes I^{\bar q}H_{n-i}(X;\Q)\to \Q.$$ In the most well-known case, if $X$ is a $4n$-dimensional Witt space, which implies that $I^{\bar m}H_*(X;\Q)\cong I^{\bar n}H_{*}(X;\Q)$ for the lower-middle and upper-middle perversities $\bar m$ and $\bar n$, then one obtains a symmetric middle-dimensional self-pairing $I^{\bar m}H_{2n}(X;\Q)\otimes I^{\bar m}H_{2n}(X;\Q)\to \Q$, and hence a signature invariant. This is the well-known Witt signature. More generally, though, it is possible to define signatures on any 
closed oriented $4n$-dimensional pseudomanifold as follows:
If $\bar p+\bar q=\bar t$ and $\bar p(k)\leq \bar q(k)$ for all $k$, then there is a map $I^{\bar p}H_*(X;\Q)\to I^{\bar q}H_*(X;\Q)$, and this induces a nonsingular symmetric pairing on $\im(I^{\bar p}H_{2n}(X;\Q)\to I^{\bar q}H_{2n}(X;\Q))$ (see Section \ref{S: perv sig} for full details). We refer to signatures of such pairings as \emph{perverse signatures} $\sigma_{\bar p\to \bar q}(X)$ and note that the Witt space signature is a special case. Similarly, in analogy with the case for manifolds, there is also a signature on compact oriented pseudomanifolds with boundary with notation $\sigma_{\bar p\onto \bar q}(X)$.

Our main results are to extend to this setting the famous Novikov additivity and Wall non-additivity theorems. In particular we have the following (which occurs below as Theorem \ref{T: Wall}):

\begin{theorem}\label{T: wall intro}
 Let $Z\subset X$ be a bicollared codimension one subpseudomanifold of the closed oriented $4n$-pseudomanifold $X$ such that $X=Y_1\cup_Z Y_2$ and $\bd Y_1=Z=-\bd Y_2$, accounting for orientations. Then
$$\sigma_{\bar p\to\bar q}(X)=\sigma_{\bar p\onto\bar q}(Y_1)+\sigma_{\bar p\onto\bar q}(Y_2)+ \sigma(V;A,B,C).$$ 
\end{theorem}
Here, the term $\sigma(V;A,B,C)$ is a certain Maslov index that generalizes Wall's correction term to Novikov additivity for manifolds with boundary. 
The vector space $V$ is a  ``relative perversity'' intersection homology group $I^{\bar q/\bar p}H_{2n}(Z;\Q)$ equipped with an anti-symmetric linking-type pairing. These are essentially the ``peripheral invariants'' of \cite{CS}, and they will be discussed in more detail in Section \ref{S: pairing}. The subspaces $A, B, C$ are defined as follows:  $A=\ker(i_{Z\subset Y_1}:I^{\bar q/\bar p}H_{2n}(Z;\Q)\to I^{\bar q/\bar p}H_{2n}(Y_1;\Q))$ and $C=\ker(i_{Z\subset Y_2}:I^{\bar q/\bar p}H_{2n}(Z;\Q)\to I^{\bar q/\bar p}H_{2n}(Y_2;\Q))$, both induced by inclusions of subspaces, while $B=\ker(d: I^{\bar q/\bar p}H_{2n}(Z;\Q)\to I^{\bar p}H_{2n-1}(Z;\Q))$, where $d$ is the boundary map of a long exact sequence. 
These details will be explained more fully below. However, we do note one significant corollary:

\begin{corollary}\label{C: Intro Novikov}
With the hypotheses of the preceding theorem, suppose in addition that $I^{\bar p}H_{2n}(Z;\Q)\to I^{\bar q}H_{2n}(Z;\Q)$ is surjective and $I^{\bar p}H_{2n-1}(Z;\Q)\to I^{\bar q}H_{2n-1}(Z;\Q)$ is injective (for example if $I^{\bar p}H_*(Z;\Q)\cong I^{\bar q}H_*(Z;\Q)$). Then $$\sigma_{\bar p\to\bar q}(X)=\sigma_{\bar p\onto\bar q}(Y_1)+\sigma_{\bar p\onto\bar q}(Y_2),$$ as in Novikov's additivity theorem. In particular, Novikov additivity holds if $Z$ is a manifold with trivial stratification.
\end{corollary}

We will also generalize these results to pseudomanifolds glued along partial  boundaries in Corollary
\ref{C: cor}, and we will show that Wall's theorem for manifolds with boundary follows as a consequence of Theorem \ref{T: wall intro} in Corollary \ref{cor:wall}.

Throughout the paper we will work with PL intersection homology rather than the sheaf-theoretic versions. While the relevant pairings and signatures could be obtained through sheaf-theoretic means, the PL category seems to provide the best context for the most-straightforward adaptation of Wall's arguments from \cite{Wa69}, which were also performed in the PL setting. As a nice side-benefit to this choice, many of our arguments and formulations can be visualized quite geometrically; in particular, the geometric formulation of the  relationship between the relative perversity intersection homology pairing and the intersection pairing on the boundary of a manifold is particularly pleasing, as we shall see in Section \ref{S: manifold}.

\paragraph{Motivation.}

There are several motivations for this work, aside from the general motivation of extending
the signature package to singular spaces.  One of these motivations comes from Sen's 
conjecture and related conjectures arising in string theory.  These are conjectures about
the signatures of certain $4n$-dimensional noncompact manifolds arising as moduli spaces
of particles, such as (n+1)-monopoles in the case of Sen's original conjecture.  In the 4-
dimensional cases, for which the conjecture has been proved, the signature turns out to be
the perverse signature of a compactification of the moduli space as a stratified space \cite{HHM}.
From analytic considerations, it seems likely that this will be true more generally, which
leaves still the question of how to calculate these perverse signatures to resolve the 
conjecture.  Our additivity and non-additivity results give a tool for this.  It would also be
interesting to compare the topological obstruction to additivity for perverse signatures in this paper
to the analytic obstruction to the Mayer-Vietoris techniques for reduced cohomology,
which were also motivated by Sen's conjecture and are related to perverse signatures, 
developed in some of the same settings by Carron in \cite{Car1}, \cite{Car2}, \cite{Car3}.

A second motivation comes from global analysis and PDEs.  For manifolds with 
boundary, the Maslov triple index term in Wall's non-additivity formula 
has been interpreted analytically in the context of analytic signature theorems for manifolds
with corners of codimension two in \cite{HMM} and in terms of a gluing formula
for the $\eta$-invariant and the spectral flow for operators with varying 
boundary conditions in \cite{KL2}.  It seems very likely, therefore that our non-additivity
formula will also turn out to relate to analytic signature theorems for pseudomanifolds
with boundary and signature gluing theorems for pseudomanifolds.  In particular,
although a signature theorem has been proved for manifolds with cusp-bundle
ends in \cite{Va}, and has been interpreted in terms of perverse signatures for
pseudomanifold compactifications of these spaces in \cite{HHM},
there is as yet no analytic signature or signature gluing theorem for manifolds 
with cusp-edge corners.  This is an interesting analytic case to tackle, and having a sense
of what should arise from the topology is helpful in doing this.

A third motivation comes from spectral sequences of perverse sheaves.  In \cite{CD}
and \cite{Hun07}, the difference between various perverse signatures in the case of a 
pseudomanifold with only two strata was interpreted in terms of a signature
on the pages of the Leray spectral sequence of the fibration on the unit normal 
bundle of the singular stratum.  It should be possible to interpret the difference between
perverse signatures for a general pseudomanifold in terms of the pages of 
the hypercohomology spectral sequence for perverse sheaves near the lower
strata.  

Finally, a fourth motivation is a Wall-type non-additivity result for Witt spaces and
possibly also for the new more general signature theory introduced by Banagl
in \cite{BaIH}.
Intersection homology of pseudomanifolds was developed
in the late 1970s and early 1980s, through the work of McCrory \cite{McC}, 
Cheeger \cite{Chee80}, and  Goresky and MacPherson \cite{GM1}.
Intersection homology groups for a pseudomanifold are parametrized by a 
function called a perversity.
There is a subclass of stratified spaces, called Witt spaces, for which there
is a Poincar\'{e} dual `middle perversity' intersection homology, and for
$4n$-dimensional Witt spaces, it is therefore possible to define a `middle
perversity signature'.  Most of the signature package has been generalized
to Witt spaces.  In particular, the Witt cobordism group has been computed and the invariance
of signature under Witt cobordism was proved by Siegel \cite{Si83} in 1983.
In the same paper, he proved a version of Novikov additivity for Witt spaces
where the dividing hypersurface is again Witt.  In a very recent paper, \cite{ALMP},
progress has also been made on the analytic side of the signature
package for Witt spaces.  In particular, the authors prove that the topological
middle perversity signature for Witt spaces is the signature of the unique extension
of the signature operator for the spaces endowed with iterated cone metrics.
The signature on Witt spaces is a particular case of a perverse signature, so
our theorem generalizes Siegel's additivity theorem to a Wall-type nonadditivity
theorem for these spaces.

Banagl has extended signature theory further to a class of ``non-Witt'' spaces (despite the terminology, this class of spaces includes all Witt spaces); these spaces are defined in terms of certain  signature conditions on the neighborhoods of odd codimensional strata. If a non-Witt space is actually Witt, Banagl's signature agrees with the Witt space signature. 
It seems possible that Banagl's signature may in fact always be a perverse
signature.  Our (non-)additivity results may help determine if this is true, and, 
if so, gives an additivity and non-additivity result for Banagl's signatures. Levikov \cite{Lev07, Lev10} has proven a Novikov additivity theorem for Banagl's signatures in a certain special case involving a union along a manifold; this is consistent with our hypothesis via Corollary \ref{C: Intro Novikov}.

\paragraph{Outline.} In order to generalize Wall's theorem to perverse signatures, we first need
to review past results and make some new definitions.  In the next section, we review signatures for manifolds, and in the following section we review intersection homology and make some new constructions.  In  Section \ref{S: Wall},
we prove our non-additivity result, obtaining as a corollary our additivity
theorem.  We prove it first for stratified pseudomanifolds without boundary,
then generalize to those with boundary.
In Section \ref{S: manifold}, we discuss the relationships of our work to  Wall's original theorem and give
two examples of calculations.
Finally, in an appendix we carefully establish some conventions regarding orientation
and intersection numbers that we use in the paper.

\medskip

\paragraph{Acknowledgment.} The authors would like to thank Markus Banagl for several helpful discussions, as well as the anonymous referees of a previous version of the paper.

\section{Background on signatures and (non-)additivity}

In this section and the following, we recall known results concerning signatures and provide a crash course on the relevant version of intersection homology.

\subsection{Additivity and non-additivity}

Recall that the signature of a closed connected oriented $4n$-manifold is the signature $\sigma(M)$ of the nondegenerate symmetric intersection pairing $$\pfb:H_{2n}(M;\Q)\otimes H_{2n}(M;\Q)\to \Q,$$ i.e. $\sigma(M)$ is the dimension of the largest positive definite subspace of this pairing minus the dimension of the largest negative definite subspace. Alternatively, this is the same as the signature of cup product pairing $H^{2n}(M;\Q)\otimes H^{2n}(M;\Q)\to H^{4n}(M;\Q)\cong \Q$ or the signature of the pairing given by exterior product of forms in de Rham cohomology $H^{2n}(M;\R)\otimes H^{2n}(M;\R)\to H^{4n}(M;\R)\cong \R$.

If $N$ is a manifold with boundary, we instead have a nondegenerate intersection pairing $$\pfb:H_{2n}(N;\Q)\otimes H_{2n}(N,\bd N;\Q) \rightarrow \Q.$$  This descends to a nondegenerate
symmetric pairing on $$\pfc:\im( H_{2n}(N;\Q) \rightarrow H_{2n}(N,\bd N;\Q)),$$ where the arrow
is induced by inclusion.  The signature of this pairing 
 is the signature $\sigma(N)$.

Suppose now that $M$ is a closed, oriented $4n$-manifold, and that $M = M_1 \cup_Z M_2$, where $M_1,M_2$ are manifolds-with-boundary oriented compatibly with $M$ and $Z=\bd M_1=-\bd M_2$.    The Novikov additivity theorem for the signature of compact $4n$-manifolds is:

\begin{theorem}[Novikov]
\label{th:novikov}
\[
\sigma(M) = \sigma(M_1) + \sigma(M_2).
\]
\end{theorem}

Since signature theory of compact manifolds is nontrivial (i.e. there exist manifolds with non-zero signature), the theory of signatures of manifolds
with boundary must also, by Novikov additivity, be nontrivial. It also turns out to be more subtle.
  The Atiyah-Patodi-Singer index theorem, \cite{APS}, 
showed that the signature of a manifold with boundary may be realized as the index
of the signature operator with a certain global boundary condition \cite{BBW}, but
that it differs from the $L$-genus of the manifold by a spectral invariant of the boundary
called the $\eta$-invariant.  It is also clear that signature for manifolds with boundary 
cannot satisfy a general Novikov additivity, as any manifold may be broken up into 
pieces that are homeomorphic to a disk, which has trivial signature.  In \cite{Wa69},
Wall identified the defect in additivity for signatures of manifolds with boundary in terms
of the Maslov triple index:

\begin{figure}[!htp]
\begin{center}
\scalebox{.4}{\includegraphics{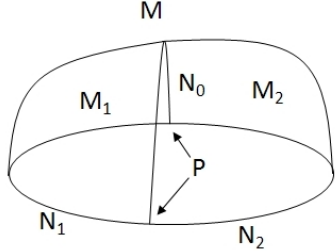}}
\caption{A schematic of the hypothesis of Wall's theorem: The manifold-with-boundary $M$ is split into the pieces $M_1$ and $M_2$ along the hypersurface $N_0$. The boundary of $M$ is split into $N_1$ and $N_2$ along $P$.}\label{F: fig4}
\end{center}
\end{figure}

\begin{theorem}[Wall \cite{Wa69}]
\label{wall}
Suppose 
$M^{4n}$ is a compact oriented manifold with boundary such that $M=M_1\cup M_2$, where $M_1,M_2$ are compact oriented manifolds with boundary. Let $N_1=\bd M\cap M_1$ and $N_2=\bd M\cap M_2$. Suppose $M_1\cap M_2=\bd M_1\cap\bd M_2$ is a  manifold $N_0$ with boundary  such that $\bd M_1=N_0\cup -N_1$, $\bd M_2=N_2\cup -N_0$, and $\bd N_1=\bd N_2=\bd N_0=P$ (see Figure \ref{F: fig4}). Then 
\begin{equation}
\label{eq:wall}
\sigma(M) = \sigma(M_1) + \sigma(M_2) - \sigma(V;A,B,C),
\end{equation}
where $\sigma(V;A,B,C)$ is the Maslov triple index for the symplectic
vector space (under the intersection pairing) $V=H_{2n-1}(P;\Q)$ with respect to the three Lagrangian
subspaces $A = \Ker(H_{2n-1}(P;\Q) \rightarrow H_{2n-1}(N_1;\Q))$, 
$B =  \Ker(H_{2n-1}(P;\Q) \rightarrow H_{2n-1}(N_0;\Q))$ 
and $C =  \Ker(H_{2n-1}(P;\Q) \rightarrow H_{2n-1}(N_2;\Q))$.
\end{theorem} 

We recall the definition of $\sigma(V;A,B,C)$ in the next subsection.

\subsection{Wall's Maslov index and some care with signs}\label{S: Wall1}

Here, we briefly review the algebraic version of the Maslov index presented by Wall in \cite{Wa69}. We also make some observations regarding a couple of sign issues that are not completely clear in Wall's original paper. For an expository account containing many viewpoints on the Maslov index, we refer the reader to \cite{CLM}.

Suppose $V$ is a vector space over $\Q$, $\Phi:V\times V\to \Q$ is a bilinear map, and  $A,B,C\subset V$ are such that $\Phi(A\times A)=\Phi(B\times B)=\Phi(C\times C)=0$. Wall considers the space $W=\frac{A\cap(B+C)}{A\cap B+A\cap C}$ (which is isomorphic to the spaces formed by permuting $A$, $B$, and $C$). Given $a,a'\in A$ representing elements of $W$, then $a=-b-c$ and $a'=-b'-c'$ for some $b,b'\in B$ and $c,c'\in C$. It is easy to show using these relations that we must have \begin{equation}\label{E: Phi}
\Phi(b,a')=-\Phi(c,a')=\Phi(c,b')=-\Phi(a,b')=\Phi(a,c')=-\Phi(b,c').
\end{equation}
From here, one obtains a well-defined pairing $\Psi$ on $W$ by setting $\Psi(a,a')=\Phi(a,b')$. This pairing is unaltered by even permutation of $A,B,C$ and is altered by a sign for odd permutations. If $\Phi$ is skew-symmetric, $\Psi$ is symmetric, and its signature is denote $\sigma(V;A,B,C)$. 
In the statement of Wall's theorem above, $V$ is the vector space $H_{2n-1}(P;\Q)$ with its intersection pairing, and $A,B,C$ are the kernels of the various maps induced by including $P$ in $N_1$, $N_0$, and $N_2$.

In Wall's ensuing topological arguments, there are some sign issues with which one needs to take care. In the proof of his non-additivity theorem, Wall  instead uses the formulation $W=\frac{B\cap(C+A)}{B\cap C+B\cap A}$ (for appropriate choices of $A,B,C$). This cyclic permutation should not affect signs. However, Wall ultimately encounters an intersection pairing $(\bd \eta)\pfa(\bd  \xi')$ representing $\Phi(\bd \eta,\bd \xi')$, where $\bd \eta\in B$ and $\bd \xi'\in C$.  Taking $B,C,A$ in that order, this is then a pairing between an element from the first subspace and an element from the second subspace. By definition, this is $\Psi$ (whereas Wall states that  this intersection pairing represents $-\Psi$). However, the intersection   $(\bd \eta)\pfa(\bd  \xi')$ is not itself quite correct. This intersection pairing is in Wall's space $Z$ (our $P$), which is the boundary of a space $X_0$ (our $N_0$), which is itself \emph{the negative of} part of the boundary of $Y_+$ (our $M_2$). By ``negative of'', we mean with the reversed orientation. Wall at first encounters the intersection $\eta\pfa_{Y_+}\xi'$ and states this is equal to $\eta\pfa_{X_0} \bd \xi'$ in $X_0$. However,  with the conventions we establish below in the Appendix, since the degree of $\eta$ is even, $\eta\pfa_{Y_+}\xi'$ will be the negative of the intersection $\eta\pfa_{X_0} \bd \xi'$ because $X_0$ has its orientation reversed as it appears in the boundary of $Y_+$. Then from here, we do have that $\eta\pfa_{X_0} \bd \xi'=\bd \eta\pfa_{Z} \bd \xi'$. 
 Putting these sign issues together, it is correct that Wall arrives at the pairing $-\Psi$, and the statement of Wall's non-additivity theorem is  correct in \cite{Wa69}.

\section{Background and preliminaries on intersection homology}

In this section, we review intersection homology and make the necessary definitions to allow us to state our generalized non-additivity theorem.
We begin with a basic review of pseudomanifolds and intersection homology; the experts might want to skim this section, as we use some recent generalizations with which they might not be familiar. Then in the second subsection below we define perverse signatures.    We also
will need some symplectic vector space that plays the role of $H_{2n-1}(P;\Q)$
from Wall's theorem.  We define this space in the third subsection.

\subsection{Review of intersection homology}\label{S: Review}

We begin with a brief review of basic definitions. For further reference, we refer the reader to \cite{GBF26, GBF23} as the background resources most suited to the brand of intersection homology treated here:  intersection homology with general perversities and stratified coefficient systems.  Other standard sources for more classical versions of intersection homology include  \cite{GM1, GM2, Bo, KirWoo, BaIH, Ki, GBF10}. Although we will not pursue them in detail here, various analytic approaches to intersection homology can be found in, e.g \cite{Chee80, CGM82, Bry, BHS}; 
these are particularly useful for
relating intersection homology to $L^2$-cohomology and harmonic forms, as in \cite{Chee80}, \cite{HHM}, \cite{SS},
and others, and for relating perverse signatures
to $L^2$-signatures, as in \cite{Dai}, \cite{Hun07}.

\paragraph{Stratified pseudomanifolds.}

We use the definition of stratified pseudomanifold in \cite{GM2}, except that
we allow strata of codimension one. 
Before recalling  the definition we need some background.

For a space $W$, we define the {\it open cone} $c(W)$ by $c(W)=([0,1)\times 
W)/(0\times W)$ (we put the $[0,1)$ factor first so that our signs will 
be consistent with the usual definition of the algebraic mapping cone). Note 
that $c(\emptyset)$ is a point.

If $Y$ is a filtered space
$$Y=Y^n\supseteq Y^{n-1}\supseteq \cdots \supseteq 
Y^0\supseteq Y^{-1}=\emptyset,$$ 
we let $c(Y)$ to be the filtered space with
$(c(Y))^i=c(Y^{i-1})$ for $i\geq 0$ and $(c(Y))^{-1}=\emptyset$.

The definition of stratified pseudomanifold is now given by induction on the
dimension. 

\begin{definition}\label{D: pseudomanifold}
A $0$-dimensional stratified pseudomanifold $X$ is a  discrete set of points 
with the trivial filtration $X=X^0\supseteq X^{-1}=\emptyset$.

An $n$-dimensional \emph{(topological) stratified  pseudomanifold}
$X$ is a paracompact Hausdorff space together 
with a filtration by closed subsets

\begin{equation*}
X=X^n\supseteq X^{n-1} \supseteq X^{n-2}\supseteq \cdots \supseteq X^0\supseteq X^{-1}=\emptyset
\end{equation*}
such that
\begin{enumerate}
\item $X-X^{n-1}$ is dense in $X$, and
\item for each point $x\in X^i-X^{i-1}$, there exists a neighborhood
$U$ of $x$ for which there is a  \emph{compact} $n-i-1$ dimensional 
stratified    pseudomanifold  $L$ and a   homeomorphism
\begin{equation*}
\phi: \R^i\times cL\to U
\end{equation*}
that takes $\R^i\times c(L^{j-1})$ onto $X^{i+j}\cap U$. A neighborhood $U$ with
this property is called {\it distinguished} and $L$ is called a {\it link} of
$x$.
\end{enumerate}
\end{definition}

The $X^i$ are called \emph{skeleta}. We write $X_i$ for $X^i-X^{i-1}$; this 
is an $i$-manifold that may be empty. We refer to the connected components of 
the various $X_i$ as  \emph{strata}\footnote{This terminology agrees with some 
sources, but is slightly different from others, including our own past work, 
which would refer to $X_i$ as the stratum and what we call strata as 
``stratum components.''}. If a stratum $Z$ is a subset of $X_n$ it 
is called a \emph{regular stratum}; otherwise it is called a \emph{singular 
stratum}.  The \emph{depth} of a stratified pseudomanifold is the number of 
distinct skeleta it possesses minus one.

We note that this definition of stratified pseudomanifolds is slightly more general than the one in common usage \cite{GM1}, as it is common to assume that $X^{n-1}=X^{n-2}$. We will not make that assumption here, but when we do assume $X^{n-1}=X^{n-2}$, intersection homology with Goresky-MacPherson perversities is known to be a topological invariant; in particular, it is invariant under choice of  stratification (see \cite{GM2}, \cite{Bo}, \cite{Ki}). Examples of pseudomanifolds include irreducible complex algebraic and analytic varieties (see \cite[Section IV]{Bo}).

\subparagraph{Pseudomanifolds with boundary.}

In manifold theory, one considers not just manifolds, which are initially defined so that every point has a Euclidean neighborhood, but also $\bd$-manifolds, for which points might have neighborhoods homeomorphic to Euclidean half-spaces. This is the familiar notion of ``manifolds with boundary''.  Even if one's ultimate intent is to study closed manifolds (those with empty boundary), boundaries naturally arise if one attempts to cut a manifold into smaller pieces.

In this section, we provide the definition of $\bd$-stratified pseudomanifold developed in \cite{GBF25}. The notion of ``pseudomanifold with boundary'' in the context of intersection homology goes back at least to Siegel's thesis \cite{Si83}, though it is difficult to find technical formulations in the literature.

\begin{definition}\label{def boundary}
An $n$-dimensional
\emph{$\bd$-stratified pseudomanifold} (or ``$\bd$-pseudomanifold'' if we do not wish to emphasize the stratification)  is a pair $(X,B)$ together with a
filtration on $X$ such that
\begin{enumerate}
\item $X-B$, with the induced filtration, is an $n$-dimensional stratified
pseudomanifold (in the sense of Definition \ref{D: pseudomanifold}),
\item $B$, with the induced filtration, is an $n-1$ dimensional stratified
pseudomanifold (in the sense of Definition \ref{D: pseudomanifold}),
\item\label{I: collar} $B$ has an {\it open collar neighborhood} in $X$, that is, a
neighborhood $N$ with a homeomomorphism of filtered spaces $N\to
B\times [0,1)$ (where $[0,1)$ is given the trivial filtration) that takes
$B$ to $B\times \{0\}$.
\end{enumerate}
$B$ is called the 
\emph{boundary} of $X$ and denoted $\bd X$.  
\end{definition}

We will often abuse notation by referring to the ``$\bd$-stratified 
pseudomanifold $X$,'' leaving $B$ tacit. 

Note that a stratified pseudomanifold $X$ (as defined in Definition \ref{D: pseudomanifold}) is a $\bd$-stratified pseudomanifold with $\bd X=\emptyset$.  
As in classical manifold theory, if we wish to emphasize the point that a $\bd$-stratified pseudomanifold $X$ is compact with $\bd X=\emptyset$, we will refer to such a $\bd$-stratified pseudomanifold as \emph{s-closed}, where the ``s'' is meant to indicate the dependence of this property on the stratification; see below for examples. 

\begin{definition}
The {\it strata} of a $\partial$-stratified pseudomanifold $X$ are the 
components of the spaces $X^i-X^{i-1}$.
\end{definition}

It is shown in \cite{GBF25} that when there are no codimension one strata, the boundary
$\partial X$ is a topological invariant. However, this is not true if codimension one strata are 
allowed, as shown by the following example.

\begin{example1}\label{E: example}
Let $M$ be a compact  $n$-manifold with boundary (in the classical sense), and let $P$ be its manifold boundary.
\begin{enumerate}
\item
Suppose we filter $M$ trivially so that $M$ itself is the only non-empty stratum. Then 
$(M,P)$ is a $\bd$-stratified pseudomanifold. Note that all the conditions of Definition \ref{def boundary} are fulfilled: $M-P$ is an $n$-manifold, $P$ is an $n-1$ manifold, and $P$ is collared in $M$ by classical manifold theory (see \cite[Proposition 3.42]{Ha}). So in this case, the notion of boundary for a $\bd$-stratified pseudomanifold and for an unfiltered $\bd$-manifold agree.

\item
On the other hand, suppose $X$ is the filtered space  $M\supset P$. Then it is easy to check that $(X,\emptyset)$ is a $\bd$-stratified pseudomanifold; that is, $X$ is a stratified 
pseudomanifold in the sense of Definition \ref{D: pseudomanifold}. 
With this filtration, we cannot have $\bd X=P$ because condition \eqref{I: collar} of Definition \ref{def boundary} would not be satisfied. Thus with this stratification, $X$ is s-closed.
\end{enumerate}
\end{example1}

Throughout this paper, the word ``boundary'' and the symbol $\bd X$ will always refer to the pseudomanifold boundary defined here and the compatible manifold scenario from the first part of the example. In contexts where we discuss a classical $\bd$-manifold $M$ with non-trivial stratification but still wish to consider the boundary of the trivially stratified $M$, we will emphasize this by referring to the \emph{manifold boundary} of $M$.

\paragraph{Piecewise linear pseudomanifolds.}

A  \emph{piecewise linear} (or \emph{PL}) stratified pseudomanifold or $\bd$-stratified pseudomanifold is a stratified pseudomanifold or $\bd$-stratified pseudomanifold with a PL structure compatible with the filtration, meaning that each skeleton is a PL subspace, and such that each link is a compact PL stratified pseudomanifold and the distinguished neighborhood homeomorphisms $U\cong \R^{n-k}\times cL$ or $U\cong \R^{n-k}_+\times cL$ are PL homeomorphisms. In this paper, we will restrict ourselves entirely to the PL setting. This is sufficient for the purpose of analysts or algebraic geometers wishing to consider Thom-Mather or Whitney stratified spaces, which are  $\bd$-pseudomanifolds. Our results should also hold for the class of topological $\bd$-stratified pseudomanifolds, but we wish to avoid the technical details we would need to pursue, such as topological general position or, alternatively, some extremely careful sheaf theory.

\paragraph{Intersection homology.}
We will work mostly with PL chain intersection homology theory with general perversities and stratified coefficient systems. General perversities (those not necessarily satisfying the axioms of Goresky and MacPherson \cite{GM1}) are indispensable for certain results, such as the intersection homology  K\"unneth theorem of \cite{GBF20}. Similarly, stratified coefficients are necessary in order to properly formulate the most useful version of intersection homology with general perversities. 
More detailed overviews of this version of the theory can be found in \cite{GBF26, GBF23}.

\subparagraph{General perversities.}
A  \emph{general perversity} on a $\bd$-stratified pseudomanifold $X$ is any function $\bar p: \{\text{singular strata of $X$} \}\to \Z$. It is technically convenient also to define $\bar p(Z)=0$ if $Z$ is a regular stratum of $X$.

\subparagraph{Stratified coefficient systems.}
In order to formulate the chain version of intersection homology for general perversities that seems best to fit with the classical sheaf-theoretic versions of intersection homology, we must use ``stratified coefficients,'' as introduced in \cite{GBF10} (see \cite{GBF26} for an exposition and also \cite{GBF23}). Since the situation simplifies somewhat in the PL category (and since we will not be working with local coefficient systems), we present a simpler definition here than is found elsewhere. In previous papers, the relevant chain complexes would have been denote $I^{\bar p}C_*(X;G_0)$, but we here follow the practice of \cite{GBF25} and write simply $I^{\bar p}C_*(X;G)$. However, these should not be confused with the intersection chain complexes of King \cite{Ki}.

First, recall that the PL chain complex $C_*(X;G)$ of a PL space $X$ is defined to be $\dlim_{T\in \mc T}C^T_*(X;G)$, where each $T$ is a triangulation of $X$ compatible with the PL structure and $C^T_*(X;G)$ is the corresponding simplicial chain complex with coefficients in the
abelian group $G$. The limit is taken over all triangulations compatible with the PL structure\footnote{It is technically necessary to work with such chains in discussing intersection homology, since degenerate cases can occur if a triangulation is not sufficiently fine. However, it is also possible to work with any sufficiently fine fixed triangulation; see \cite{MV86}.}. In other words, elements of $C_*(X;G)$ are represented by sums of chains, each of which is taken from some fixed triangulation of $X$. In particular, any $\xi\in C_j(X;G)$ can be represented as a finite sum $\xi=\sum g_i\sigma_i$, where $g_i\in G$ and $\sigma_i$ is a $j$-simplex in some triangulation of $X$. Furthermore, $\bd \xi=\sum g_i\bd \sigma_i$.

Now, suppose $X$ is a $\bd$-stratified pseudomanifold. We define $C_j(X;G)_0$ to be the subgroup of $\xi\in C_j(X;G)$ such that when we write $\xi$ as $\sum g_i\sigma_i$, no $\sigma_i$ is contained in $X^{n-1}$. It is easy to check that this is a $G$-module. In order to define $C_*(X;G)_0$ as a chain complex, we define $\bd \xi$ to be $\sum g_i\bd \sigma_i- \sum_{\sigma_i\subset X^{n-1}} g_i\sigma_i$. In other words, to obtain $\bd \xi\in C_*(X;G)_0$, we remove from $\bd \xi\in C_*(X;G)$ those simplices contained in $X^{n-1}$. This is a chain complex, and we denote its homology $H_*(X;G)$. Some readers will notice that $C_*(X;G)_0$ is isomorphic to $C_*(X,X^{n-1};G)$, but for the upcoming definition of intersection homology, we require this formulation. We refer to $H_*(X;G)$ as homology with \emph{stratified coefficients} $G$. 

\subparagraph{Intersection homology.}

Given a $\bd$-stratified pseudomanifold $X=X^n$, a general perversity $\bar p$, and an abelian group $G$, one defines the \emph{intersection chain complex} $I^{\bar p}C_*(X;G)$ as a subcomplex of $C_*(X;G)_0$ as follows: An $i$-simplex $\sigma$ in $X$  is \emph{$\bar p$-allowable} if 
$$\dim(\sigma\cap Z) \leq i-\codim(Z)+\bar p(Z)$$ 
for any singular stratum  $Z$ of $X$. 
The chain $\xi\in C_i(X;G)_0$ is $\bar p$-allowable if each simplex with non-zero coefficient in $\xi$ or in $\bd \xi$ is allowable. Notice that  simplices that disappear from the boundary because of the coefficient system $G$ do not need to be checked for allowability. Notice that this is also why it is not sufficient to work in $C_*(X,X^{n-1};G)$, where we have no control over simplices that live in $X^{n-1}$. 
Let $I^{\bar p}C_*(X;G)$ be the complex of $\bar p$-allowable chains.  The associated homology theory is denoted $I^{\bar p}H_*(X;\mc G)$.

Relative intersection homology is defined similarly, in the obvious way, though we note that 
 the filtration on a subspace will always be that inherited from the larger space by restriction, i.e. if $Y\subset X$, then $Y^k=Y\cap X^k$, regardless of the actual dimensions involved. We also assume that $Y$ inherits the formal dimension of $X$, regardless of actual geometric dimension, so that if $Z$ is a stratum of codimension $k$ in $X$, then we consider $Z\cap Y$ to have the same codimension $k$ in $Y$.
  Thus a chain  in $Y$ is defined to be allowable if and only if it is allowable in $X$. 
 
If $\bar p$ is a perversity in the sense of Goresky-MacPherson \cite{GM1} and $X$ has no strata of codimension one, then $I^{\bar p}H_*(X;G)$ is isomorphic to the intersection homology groups of Goresky-MacPherson \cite{GM1, GM2}. If $\bar p$ is not a Goresky-MacPherson perversity, then we need stratified coefficients in order for some of the main properties of intersection homology, such as duality and the cone formula, to hold - see \cite{GBF23, GBF26}. General perversities are useful because, among other things, they allow us
to talk about relative and absolute cohomologies in the same framework as the Goresky-MacPherson
intersection homologies.  Suppose that $Z \subset X$ are smooth manifolds. Then  if $\bar p(Z)>\codim(Z)-2$ we get
 $I^{\bar p}H_*(X;G)\cong H_*(X,Z;G)$, and if $\bar p(Z)<0$, 
we get  $I^{\bar p}H_*(X;G)\cong H_*(X-Z;G)$. Note that if $Z$ is the manifold boundary of $X$, then also $H_*(X-Z;G)\cong H_*(X;G)$.

Intersection homology with general perversities can also be formulated sheaf theoretically \cite{GBF23, GBF26}  or analytically \cite{Sa05}. In these languages, it is more customary to use cohomological indexing and refer to intersection \emph{cohomology} but these are really the same theories (up to various indexing issues).

Even with general perversities and $G$ coefficients, the basic properties of $I^{\bar p}H_*(X;G)$ established in  \cite{Ki} and \cite{GBF10} hold with little or no change to the proofs, such as stratum-preserving homotopy equivalence, excision, the K\"unneth theorem for which one term is an unstratified  manifold, Mayer-Vietoris sequences, etc. For more details of this construction (and more general cases), see \cite{GBF26, GBF10,GBF23}.

\paragraph{Intersections and duality.} Finally, we recall the intersection homology version of Poincar\'e duality, due initially to Goresky and MacPherson \cite{GM1} and later extended to the more general cases considered here \cite{GBF23}. Suppose $X$ is an s-closed oriented $n$-pseudomanifold and that $F$ is a field. Suppose that $\bar p$ and $\bar q$ are perversities such that $\bar p+\bar q\leq \bar  r$, i.e. $\bar p(Z)+\bar q(Z)\leq \bar r(Z)$ for all singular strata $Z$. Then there is a partially defined intersection pairing $$\pfa: I^{\bar p}C_i(X;F)\times I^{\bar q}C_j(X;F)\to I^{\bar r}C_{i+j-n}(X;F),$$ defined on pairs of chains $x\times y$ such that $x$ and $y$ are in stratified general position (see \cite{GM1,GBF23}). Since all pairs of homology classes can be represented by pairs of chains in stratified general position, it can be shown that the intersection pairing induces a fully-defined map on intersection homology
$$I^{\bar p}H_i(X;F)\otimes I^{\bar q}H_j(X;F)\to I^{\bar r}H_{i+j-n}(X;F).$$ If $\bar p+\bar q= \bar  t$, i.e. $\bar p(Z)+\bar q(Z)=\codim(Z)-2$ for all singular strata $Z$, and we compose with the augmentation $\epsilon: I^{\bar t}H_0(X;F)\to F$, then the intersection pairing induces a nonsingular pairing $$\pfb: I^{\bar p}H_i(X;F)\otimes I^{\bar q}H_{n-i}(X;F)\to I^{\bar t}H_{0}(X;F)\to F,$$ whose adjoint is the duality isomorphism $I^{\bar p}H_i(X;F)\cong \Hom(I^{\bar q}H_{n-i}(X;F),F)$. If $Y$ is a compact $\bd$-stratified pseudomanifold, we similarly obtain an analogous Lefschetz-type duality $I^{\bar p}H_i(Y;F)\cong \Hom(I^{\bar q}H_{n-i}(Y,\bd Y;F),F)$ from the partially-defined chain intersection pairing $\pfa: I^{\bar p}C_i(X;F)\times I^{\bar q}C_j(X, \bd X;F)\to I^{\bar r}C_{i+j-n}(X;F)$.

\subsection{Perverse signatures}\label{S: perv sig}

Now we can define perverse signatures.

If $X$ is a PL stratified pseudomanifold and $\bar p\leq \bar q$, let $$I^{\bar p\to\bar q}H_*(X;\Q)=\im(I^{\bar p}H_*(X;\Q)\to I^{\bar q}H_*(X;\Q)),$$ where the map is induced by inclusion of chain complexes. If $(Y,Z)$ is a pair of a pseudomanifold and any subspace $Z$, let $$I^{\bar p\onto\bar q}H_*(Y,Z;\Q)=\im (I^{\bar p}H_*(Y;\Q)\to I^{\bar q}H_*(Y,Z;\Q)).$$ The reason for the double arrow in the second notation is to highlight that  $I^{\bar p\onto\bar q}H_*(Y,Z;\Q)$ is really the image of a composition of two maps taking the perversity from $\bar p$ to $\bar q$ and the space from $Y_i$ to $(Y,Z)$.

By duality, if $X$ is s-closed, oriented, and $4n$-dimensional, and if $Y$ is a compact, oriented $4n$-dimensional $\bd$-stratified pseudomanifold, then if $\bar p+\bar q=\bar t$, there are  nonsingular intersection pairings 
\begin{align}\label{A: int}
\pfb: I^{\bar p}H_{2n}(X;\Q)\otimes I^{\bar q}H_{2n}(X;\Q)\to \Q\\
\pfb: I^{\bar p}H_{2n}(Y;\Q)\otimes I^{\bar q}H_{2n}(Y,\bd Y;\Q)\to \Q.
\end{align}
These are each induced by the chain level pairing of chains in stratified general position, followed by augmentation $\epsilon: I^{\bar t}H_0(X;F)\to F$.
If also $\bar p\leq \bar q$, this induces  pairings 
\begin{align*}
\pfc:I^{\bar p\to\bar q}H_{2n}(X;\Q)\otimes I^{\bar p\to\bar q}H_{2n}(X;\Q)&\to\Q\\
\pfc:I^{\bar p\onto\bar q}H_{2n}(Y,\bd Y;\Q)\otimes I^{\bar p\onto\bar q}H_{2n}(Y,\bd Y;\Q)&\to\Q,
\end{align*} 
Explicitly, if $j_*:I^{\bar p}H_*(X;\Q)\to I^{\bar q}H_*(X;\Q)$ is induced by the inclusion of chains $j:I^{\bar p}C_*(X;\Q)\to I^{\bar q}C_*(X;\Q)$, then 
$j_*[x]\pfc j_*[y]$ is defined to be $[x]\pfb j_*[y]$, which itself is defined to be  $\epsilon[x\pfa j(y)]=\epsilon[x\pfa y]$, assuming $x$ and $y$ are representative chains for their intersection homology classes in stratified general position (which can always be achieved). The construction of the pairing on $I^{\bar p\onto\bar q}H_{2n}(Y,\bd Y;\Q)$ is completely analogous. 

\begin{lemma}
The pairings 
\begin{align*}
\pfc: I^{\bar p\to\bar q}H_*(X;\Q)\otimes I^{\bar p\to\bar q}H_*(X;\Q)&\to \Q\\
\pfc: I^{\bar p\onto\bar q}H_*(Y,\bd Y;\Q)\otimes I^{\bar p\onto\bar q}H_*(Y,bd Y;\Q)&\to \Q 
\end{align*}
are well-defined, nonsingular, and symmetric. 
\end{lemma}
\begin{proof}
We will treat the first pairing in detail. The second is handled similarly. 

Now let $j_*[x],j_*[y]\in I^{\bar p\to\bar q}H_*(X;\Q)$ with $j_*$ as above. By definition, $j_*[x] \pfc j_*[y]=[x]\pfb j_*[y]$, which itself is defined to be the augmentation of $x\pfa j(y)$, assuming $x$ and $y$ are representative chains for their intersection homology classes in stratified general position (which can always be achieved). By the arguments in \cite{GM1}, $[x]\pfb j_*[y]$ is independent of the choice of $x$ and $y$ within their respective intersection homology classes, again assuming stratified general position. To establish well-definedness of $\pfc$, we must show that $[x]\pfb j_*[y]=[z]\pfb j_*[y]$ if $j_*[x]=j_*[z]$.

By the bilinearity of $\pfb$ (arising from that of $\pfa$), it suffices to show that $[u]\pfb j_*[y]=0$ if $[u]\in I^{\bar p}H_{2k}(X;\Q)$ and $j_*[u]=0\in I^{\bar q}H_{2k}(X;\Q)$. Since $[u]\pfb j_*[y]$ is independent of the choice of cycle representing  $j_*[y]$ in $I^{\bar q}H_{2k}(X;\Q)$, we may assume we have chosen a representative  $y$ that is $\bar p$-allowable as $j_*[y]$ is in the image of $I^{\bar p}H_{2k}(X;\Q)$. Further, we may assume $u$ is a $\bar p$-allowable representative of $[u]$ and that $u,y$ are in stratified general position.  Since the chain level intersection pairing of $2k$-chains is symmetric \cite{GM1}, $u\pfa y=y\pfa u$ up to sign. But now by the same reasoning as above, the augmentation of $y\pfa u$ equals $[y]\pfb j_*[u]$, where now $[y]\in I^{\bar p}H_{2k}(X;\Q)$ and $j_*[u]\in I^{\bar q}H_{2k}(X;\Q)$. But by assumption, $j_*[u]=0\in I^{\bar q}H_{2k}(X;\Q)$. It follows that $[u]\pfb j_*[y]=0$, so $\pfc$ is well-defined.

The symmetry of the pairing $\pfc$ comes from the symmetry of the chain level pairing \cite{GM1}: if $x,y\in I^{\bar p}C_{2k}(X;\Q)$ are in stratified general position, then $x\pfa y=y\pfa x\in I^{\bar r}C_0(X;\Q)$ for any $\bar r\geq \bar p+\bar p$. In particular, it then follows that for $[x],[y]\in I^{\bar p}H_{2k}(X;\Q)$, we have $$j_*[x]\pfc j_*[x]=[x]\pfb j_*[y]=\epsilon[x\pf y]=\epsilon[y\pf x]=[y]\pfb j_*[x]=[y]\pfc [x].$$ 

Finally, to see that $\pfc$ is nonsingular, we simply note that if $j_*[y]\in I^{\bar p\to\bar q}H_{2n}(X;\Q)\subset I^{\bar q}H_{2n}(X;\Q)$ with $j_*[y]\neq 0$, then there must be an $[x]$ in  $I^{\bar p}H_{2n}(X;\Q)$ such that $[x]\pfb j_*[y]\neq 0$ by Goresky-MacPherson-Poincar\'e duality. But then, up to signs, $[x]\pfb j_*[y]=\epsilon(x\pfa y)=\epsilon(y\pfa x)=[y]\pfb j_*[x]$, where $\epsilon$ is the augmentation and $x,y$ are appropriately chosen representative chains, 
so the image $j_*[x]$ in $I^{\bar p\to\bar q}H_{2n}(X;\Q)$  must be non-zero. In particular, we see that $j_*[x]\pfc j_*[y]\neq 0$, so $\pfc$ is nonsingular.

Minor modifications of the same argument work for $I^{\bar p\onto \bar q}H_{2n}(Y,\bd Y;\Q)$, representing all elements as $\bar p$-allowable cyles in $Y$.
\end{proof}

\begin{definition}
Let the \emph{$\bar p \to \bar q$  perverse signatures} of $X$ and $Y$,
denoted $\sigma_{\bar p\to\bar q}(X)$ and $\sigma_{\bar p\onto\bar q}(Y)$, be the  respective signatures of $\pfc$ on  $I^{\bar p\to\bar q}H_*(X;\Q)$ and $I^{\bar p\onto\bar q}H_*(Y,\bd Y;\Q)$.
\end{definition}

\subsection{A relative-perversity intersection homology pairing}\label{S: pairing}

In this subsection, we will study the ``relative perversity'' intersection homology groups that provide the symplectic pairing for our non-additivity theorem. The relationship between this pairing and the usual Goresky-MacPherson intersection pairing is akin to the relationship between the intersection pairing on the manifold boundary of a manifold and the pairing in its interior. In fact, this will be made precise in Section \ref{S: manifold}, where we will show it reduces to the intersection pairing on the topological boundary of a manifold in the appropriate context.

These groups are essentially the same as the hypercohomology groups of the ``peripheral complex'' defined sheaf-theoretically\footnotemark  \, in \cite{CS}, though there Cappell and Shaneson work with particular perversities in a much more specific topological context. It is observed in \cite{CS} that the duality pairing of these groups 
(compare Lemma \ref{L: pairing}, below) follows from the Verdier duality properties of the Deligne sheaves used to define intersection homology sheaf theoretically. We have chosen instead to work with these groups from a PL-chain point of view, in keeping with the overall spirit of this paper. The chain theory also allows us to define our dual pairing using geometric intersections, which both is useful in the work that follows and makes the (anti)symmetry properties of the pairing easier to observe. We will provide full proofs in this context.
\footnotetext{Despite a close relation between the two papers, the peripheral complex of \cite{CS} appears to be somewhat different from the peripheral complex defined by Goresky and Siegel in \cite{GS83}. The peripheral complex of Cappell-Shaneson is defined to be the third term in the distinguished triangle involving sheaf complex maps of the form $\mc I^{\bar p}\mc C^*\to \mc I^{\bar q}\mc C^*$ for $\bar p\leq \bar q$, while those of Goresky-Siegel are affiliated to maps of the form $\mc I^{\bar p}\mc C^*\to \text{\emph{RHom}}(\mc I^{\bar q}\mc C^*,\D^*[n])$.}

\paragraph{Motivation.} 
To motivate  the groups we will need, let $M$ be a $\bd$-manifold with $\bd M\neq\emptyset$,  let $G$ be an abelian group, and 
consider the long exact sequence 

\begin{equation*}
\to H_i(M;G) \to H_i(M,\bd M;G)\to H_{i-1}(\bd M;G)\to .
\end{equation*}
By contrast, for a $\bd$-stratified pseudomanifold  $X$ and $\bar p\leq \bar q$, we will compare this with a long exact sequence 

\begin{equation*}
\to I^{\bar p}H_i(X;G) \to I^{\bar q}H_i(X;G)\to I^{\bar q/\bar p}H_{i}(X;G)\to ,
\end{equation*}
where $I^{\bar q/\bar p}H_{i}(X;G)$ is defined to be the homology of the quotient  
$I^{\bar q}C_{i}(X;G)/I^{\bar p}C_{i}(X;G)$. 

But suppose $M$ is a $\bd$-manifold with non-empty  boundary $\bd M$, and let $X$ denote the stratified pseudomanifold  $M\supset \bd M$. If we choose perversities $\bar p,\bar q$ such that  $\bar p(\bd M)<0$ and $\bar q(\bd M)>\bar t(\bd M)=-1$, then an easy computation (see \cite{GBF23}), shows that $I^{\bar p}H_*(X;G)\cong H_*(M;G)$ and $I^{\bar q}H_*(X;G)\cong H_*(M, \bd M;G)$. Thus we expect $I^{\bar q/\bar p}H_{i}(X;G)$ to play a role analogous to that classically played by the manifold boundary of a $\bd$-manifold, though with a dimension shift. We will make this connection with manifold boundaries even more precise in Section \ref{S: manifold}.

\paragraph{Two $\bar q/\bar p$ long exact sequences.} 
Let $X$ be an $n$-dimensional PL $\bd$-stratified pseudomanifold. We continue to assume  $G$ an abelian group  and  $\bar p, \bar q$ general perversities such that $\bar p(Z)\leq \bar q(Z)$ for all singular strata $Z\subset X$. Let  $I^{\bar q/\bar p}C_*(X; G)=I^{\bar q}C_*(X; G)/I^{\bar p}C_*(X; G)$, and let $I^{\bar q/\bar p}H_*(X; G)$ denote the corresponding homology groups. These groups first appear sheaf-theoretically in \cite[Section 5.5]{GM2}; we will refer to them as \emph{relative perversity intersection homology groups}.

Note that a cycle $x$ in $I^{\bar q/\bar p}C_i(X; G)$ is a $\bar q$-allowable chain such that $\bd x$ is $\bar p$-allowable. A homology between cycles $x_1$ and $x_2$ is provided by a $\bar q$-allowable chain $y$ such that $\bd y=x_1-x_2+p$, where $p$ is $\bar p$-allowable.

Suppose  $Y\subset X$ is a PL subspace (without restrictions) and that we let $I^{\bar r}C_*(Y; G)$ denote the subcomplex of $I^{\bar r}C_*(X; G)$ consisting of chains supported in $Y$. Note that $I^{\bar r}C_*(Y; G)$ might be $0$, for example if $Y\subset X^{n-1}$. In all of our later applications, $Y$ will itself be a $\bd$-stratified pseudomanifold and $I^{\bar r}C_*(Y; G)$ will be equal to the complex of intersection chains on $Y$ in the usual sense, so the notation should cause no undue alarm. Then we have the diagram

\begin{diagram}[LaTeXeqno]\label{E: sequences}
0&\rTo &I^{\bar p}C_*(Y; G) &\rTo^{i_{\bar p\to\bar q}}& I^{\bar q}C_*(Y; G)&\rTo^{\pi_{\bar q/\bar p}}& I^{\bar q/\bar p}C_*(Y; G)&\rTo &0\\
&&\dInto^{i_{Y\subset X}}&&\dInto^{i_{Y\subset X}}&&\dInto^{i_{Y\subset X}}\\
0&\rTo &I^{\bar p}C_*(X; G) &\rTo^{i_{\bar p\to\bar q}}& I^{\bar q}C_*(X; G)&\rTo^{\pi_{\bar q/\bar p}}& I^{\bar q/\bar p}C_*(X; G)&\rTo &0 \\
&&\dOnto^{\pi_{X,Y}}&&\dOnto^{\pi_{X,Y}}&&\dOnto^{\pi_{X,Y}}\\
0&\rTo &I^{\bar p}C_*(X,Y; G) &\rTo^{i_{\bar p\to\bar q}}& I^{\bar q}C_*(X,Y; G)&\rTo^{\pi_{\bar q/\bar p}}& I^{\bar q/\bar p}C_*(X,Y; G)&\rTo &0,
\end{diagram}
where $I^{\bar q/\bar p}C_*(Y; G)=I^{\bar q}C_*(Y; G)/I^{\bar p}C_*(Y; G)$ and
$I^{\bar q/\bar p}C_*(X,Y; G)$ is defined to be the quotient of  $I^{\bar q/\bar p}C_*(X; G)$ by  $I^{\bar q/\bar p}C_*(Y; G)$. 
The righthand vertical map between the first two rows is an injection because any chain supported in $Y$ that is $\bar p$-allowable in $X$ (and hence $0$ in $I^{\bar q/\bar p}C_*(X; G)$) will already be in $I^{\bar p}C_*(Y;G)$. Therefore, by the serpent lemma, the last row is also a short exact sequence. 
In particular, we have long exact sequences associated to the third row of this diagram:
{\footnotesize
\begin{diagram}[LaTeXeqno]\label{E: LES1}
&\rTo &I^{\bar p}H_i(X,Y; G) &\rTo^{(i_{\bar p\to\bar q})_*}& I^{\bar q}H_i(X,Y; G)&\rTo{(\pi_{\bar q/\bar p})_*}& I^{\bar q/\bar p}H_i(X,Y; G)&\rTo^d &I^{\bar p}H_{i-1}(X,Y; G)&\rTo&,
\end{diagram}}
\noindent and to the third column: 
{\footnotesize
\begin{diagram}[LaTeXeqno]\label{E: LES}
&\rTo &I^{\bar q/\bar p}H_i(Y; G) &\rTo^{(i_{Y\subset X})_*}& I^{\bar q/\bar p}H_i(X; G)&\rTo^{(\pi_{X,Y})_*}& I^{\bar q/\bar p}H_i(X,Y; G)&\rTo^\delta &I^{\bar q/\bar p}H_{i-1}(Y; G)&\rTo&.
\end{diagram}
}
Observe that the ``connecting maps'' $d$ can  be described as acting on a representative chain $x$ by taking $x$ to its boundary $\bd x$. Meanwhile,  an element of $I^{\bar q/\bar p}H_*(X,Y;G)$ can be represented by a chain $x$ such that $\bd x=a+b$ with $b$ a $\bar q$-allowable chain in $Y$ and $a$ a $\bar p$-allowable chain in $X$, and $\delta x$ is represented by $b$.
Note also that a representative $x$ of a class in  $I^{\bar q/\bar p}H_*(X,Y; G)$ is a $\bar q$-allowable chain on $X$
such that $\bd x$ is the sum of a $\bar p$-allowable chain on $X$ and a $\bar q$-allowable
chain on $Y$.  We will use this fact in the proof of our main theorem.

\paragraph{A pairing on relative perversity intersection homology.}
Now we want to define intersection pairings on our relative perversity intersection homology groups. No doubt the pairing we are about to introduce can be derived from abstract sheaf machinery via the Verdier duality properties of Deligne sheaves, but it will be useful for us to have a concrete geometric description, especially as this will make evident the needed (anti)symmetry properties of the pairing.

Suppose $\bar p+\bar q\leq \bar r$ and $R$ is a ring. We define a pairing $\Phi: I^{\bar q/\bar p}H_i(X;R)\otimes I^{\bar q/\bar p}H_j(X,\bd X;R)\to I^{\bar r}H_{n-i-j+1}(X;R)$ as follows. Let $\pfa$ denote the Goresky-MacPherson intersection pairing on intersection chains. If $x,y$ are chains in stratified general position representing respective elements of $I^{\bar q/\bar p}H_i(X;R)$ and $I^{\bar q/\bar p}H_j(X,\bd X;R)$, let $\tilde{\Phi}(x,y)=x\pfa \bd y+(-1)^{n-|x|} (\bd x)\pfa y\in I^{\bar r}C_{n-i-j+1}(X;R)$, where
$|x|$ denotes the degree of $x$. 
We need to see that $\tilde{\Phi}$ makes sense as a map on chains, and then
we want to show it descends to a well-defined pairing $\hat \Phi([x],[y])$ on homology.

To make sense on chains, we need to know that $x\pfa \bd y+(-1)^{n-|x|} (\bd x)\pfa y$ is an $\bar r$-allowable chain.  It follows from the standard stratified general position arguments \cite{Mc78, GM1,GBF18} that we can choose $x$ and $y$ in stratified general position (which includes boundaries being in stratified general position with respect to $x$, $y$, and each other), and we may also assume that $x$ does not intersect $\bd X$. The Goresky-MacPherson intersection pairing extends to the relevant chains in this setting by \cite{GBF18}. Note that each intersection $x\pfa \bd y$ or $(\bd x)\pfa y$ is between a $\bar p$-allowable chain and a $\bar q$-allowable chain; to see this in the case of $x\pfa \bd y$, we should observe that $\bd y$ is the sum of a $\bar p$-allowable chain in $X$ and a $\bar q$-allowable chain in $\bd X$, but the part in $\bd X$ does not intersect $x$, which can be assumed to lie in the interior of $X$, so we have the intersection of a $\bar q$-admissible chain with a 
$\bar p$-admissible chain, which will be $\bar r$-admissible. 
Furthermore\footnote{We use the sign conventions of Dold \cite{Dold} or Goresky-MacPherson \cite{GM1} so that $\bd (a\pfa b)=(\bd a)\pfa b+(-1)^{n-|a|}a\pfa(\bd b)$.}, 
\begin{align*}
\bd \tilde{\Phi}(x,y)&=(\bd x)\pfa (\bd y)+(-1)^{n-|x|+n-|x|-1} (\bd x)\pfa (\bd y)\\
&=(\bd x)\pfa (\bd y)-(\bd x)\pfa (\bd y)\\
&=0,
\end{align*} 
so indeed we obtain an admissible $\bar r$-cycle. It is important to note that, despite appearances, this cycle is  not necessarily the boundary $(-1)^{n-|x|} \bd(x\pfa y)$ as $x\pfa y$ is the 
intersection of two $\bar q$ admissible chains, thus is not necessarily well-defined in $I^{\bar r}C_*(X;R)$ unless $\bar q+\bar q\leq r$.

To show that this pairing is well-defined on homology, suppose that $z$ is another chain representing the same class as $x$ and in stratified general position with respect to $y$. Then from the definitions, $z-x=\bd Q+P$, where $Q$ is another $\bar q$-allowable chain whose boundary is $\bar p$-allowable and $P$ is $\bar p$-allowable. Again, we can assume everything in stratified general position and that $P$ and $Q$ do not intersect $\bd X$. Then 

\begin{align*}
\tilde{\Phi}(z,y)&=z\pfa \bd y+(-1)^{n-|z|} (\bd z)\pfa y\\
&=(x+\bd Q+P)\pfa \bd y+(-1)^{n-|z|} (\bd (x+\bd Q+P))\pfa y\\
&=x\pfa \bd y +(-1)^{n-|z|} (\bd x)\pfa y + (\bd Q+P)\pfa \bd y+(-1)^{n-|z|} (\bd P)\pfa y\\
&=x\pfa \bd y +(-1)^{n-|z|} (\bd x)\pfa y + P\pfa \bd y+(-1)^{n-|z|} (\bd P)\pfa y +\bd Q\pfa \bd y.
\end{align*}
Note that each intersection is of a $\bar p$-allowable chain with a $\bar q$-allowable chain since $\bar p$-allowable chains are also $\bar q$-allowable. Now, since $|z|=|P|$, we see that $P\pfa \bd y+(-1)^{n-|z|} (\bd P)\pfa y=(-1)^{n-|z|}\bd (P\pfa y)$, which is well defined because $P$ is $\bar p$-allowable and $y$ is $\bar q$-allowable. Similarly, $\bd Q\pfa \bd y=\bd (Q\pfa \bd y)$, using again that only the $\bar p$-allowable part of $\bd y$ can intersect $Q$. Thus $\tilde{\Phi}(z,y)=\tilde{\Phi}(x,y)$. 
A similar argument shows that the pairing is independent of the choice of chain representing $[y]$,
so $\hat \Phi([x],[y])$ is well-defined in  $I^{\bar r}H_{n-i-j+1}(X;R)$.

Now, let $Z$ be a compact oriented $4n-1$ PL $\bd$-stratified pseudomanifold. Suppose $\bar p+\bar q=\bar t$. Then the composition of  $\hat \Phi: I^{\bar q/\bar p}H_{i}(Z;\Q) \otimes I^{\bar q/\bar p}H_{4n-i}(Z,\bd Z;\Q)\to I^{\bar t}H_0(Z;\Q)$ with the augmentation $\epsilon: I^{\bar t}H_0(Z;\Q)\to \Q$  gives us a bilinear form $\Phi: I^{\bar q/\bar p}H_{i}(Z;\Q) \otimes I^{\bar q/\bar p}H_{4n-i}(Z,\bd Z;\Q)\to \Q$.

\begin{lemma}\label{L: pairing}
$\Phi: I^{\bar q/\bar p}H_{i}(Z;\Q) \otimes I^{\bar q/\bar p}H_{4n-i}(Z,\bd Z;\Q)\to \Q$ is nonsingular, and if $\bd Z=\emptyset$, it is skew-symmetric on   $I^{\bar q/\bar p}H_{2n}(Z;\Q)$.
\end{lemma}
\begin{proof}
Consider  the following diagram with coefficients in $\Q$:

{\footnotesize
\begin{diagram}
&\rTo &I^{\bar q}H_{i}(Z) &\rTo&I^{\bar q/\bar p}H_{i}(Z)&\rTo & I^{\bar p}H_{i-1}(Z)&\rTo \\
&&\dTo&&\dTo&&\dTo\\
&\rTo &\Hom(I^{\bar p}H_{4n-1-i}(Z,\bd Z),\Q)&\rTo
& \Hom(I^{\bar q/\bar p}H_{4n-i}(Z,\bd Z),\Q)&\rTo& \Hom(I^{\bar q}H_{4n-i}(Z,\bd Z),\Q) &\rTo.
\end{diagram}}
The top is the long exact sequence induced by the short exact sequence
\begin{diagram}
0&\rTo& I^{\bar p}C_{*}(Z;\Q)&\rTo &I^{\bar q}C_*(Z;\Q) &\rTo&I^{\bar q/\bar p}C_{*}(Z;\Q)&\rTo 0.
\end{diagram}
The bottom is the $\Hom(\cdot,\Q)$ dual of the same long exact sequence for $(Z,\bd Z)$; it is also exact because $\Q$ is a field. The first and third vertical morphisms take a class $[x]$ to $[x] \pfb \cdot $. By Goresky-MacPherson \cite{GM1,GM2} (and \cite{GBF23} for general perversities and $\bd$-stratified pseudomanifolds), these are isomorphisms.  The second vertical map takes $[x]$ to $\Phi([x],\cdot)$. 

We claim that the diagram commutes up to sign.
It is  standard that the square with top corners $I^{\bar p}H_{i}(Z)$ and $I^{\bar q}H_{i}(Z)$ 
(not shown as a square on the diagram) commutes.
To see that the first square commutes, let $[x]\in I^{\bar q}H_{i}(Z)$, $[z]$ be the image of $[x]$ in $I^{\bar p/\bar q}H_{i}(Z)$, and $[y]\in I^{\bar q/\bar p}H_{4n-i}(Z,\bd Z)$.  Then $\Phi([z],[y])=\epsilon[z\pfa \bd y+(-1)^{4n-1-|x|}\bd z\pfa y]=\epsilon[z\pfa \bd y]=\epsilon[x\pfa \bd y]=[x]\pfb[y]$ because $\bd x=\bd z=0$ and $[z]$ and $[x]$ are represented by the same chain. On the other hand, going down then right takes $[x]$ to a map that acts on $[y]$ by first applying the map to the left on homology that gives $[\bd y]$ and then applying $[x]\pfb \cdot$. So this square also commutes. 

To see that the second square commutes up to sign, suppose $[x]\in I^{\bar q/\bar p}H_{i}(Z)$. Then going right then down takes $[x]$ first to $[\bd x]$, then to the map that acts on $[y]\in I^{\bar q}H_{4n-i+1}(Z,\bd Z)$ by $[\bd x]\pfb [y]$ (note that $\bd y$ is supported in $\bd Z$ and cannot intersect $x$, which can be assumed to have support in the interior of $Z$). On the other hand, going down then right takes $[x]$ to the map that acts on $[y]\in I^{\bar q}H_{4n-i+1}(Z,\bd Z)$ by first taking it to $[y] \in I^{\bar q/\bar p}H_{4n-i}(Z,\bd Z)$ and then applying $\Phi([x],\cdot)$ to obtain $\epsilon[x\pfa \bd y+(-1)^{4n-1-|x|}(\bd x)\pfa y]$. But again $\bd y$ must lie in $\bd Z$, so this is just $\epsilon[(-1)^{4n-1-|x|}(\bd x)\pfa y]=(-1)^{4n-1-|x|}[\bd x]\pfb [y]$.

We can now apply the five lemma to conclude that $\Phi$ determines a nonsingular pairing. Even though the diagram does not commute on the nose, commuting up to sign implies that it is possible to change signs of some of the maps to obtain a commuting diagram. Changing signs does not affect exactness of the horizontal sequences.

To show $\Phi$ is anti-symmetric when $i=2n$ and $\bd Z=\emptyset$, we calculate \footnote{Recall that on an $m$-dimensional $\bd$-stratified pseudomanifold   $a\pfa b=(-1)^{(m-|a|)(m-|b|)}b\pfa a$; see \cite{Dold, GM1}.\label{fn}},
\begin{align*}
\hat \Phi([x],[y])&=[x\pfa (\bd y)+(-1)^{4n-1-|x|}(\bd x)\pfa y]\\
&=[x\pfa (\bd y)-(\bd x)\pfa y]\\
&= [(-1)^{(4n-1-|x|)(4n-1-(|y|-1))} (\bd y)\pfa x- (-1)^{(4n-1-|y|)(4n-1-(|x|-1))}y\pfa (\bd x)]\\
&= [(-1)^{(4n-1-2n)(4n-1-(2n-1))} (\bd y)\pfa x- (-1)^{(4n-1-2n)(4n-1-(2n-1))}y\pfa (\bd x)]\\
&= [ (\bd y)\pfa x- y\pfa (\bd x)]\\
&= -\hat \Phi([y],[x]).
\end{align*}
\end{proof}

In Section \ref{S: manifold}, below, we will show that if $X$ is a $4n-1$-manifold with non-empty manifold boundary, appropriately stratified and with an appropriate choice of perversities, then $\Phi$ becomes the classical intersection pairing on $\bd X$.

\section{Non-additivity of perverse signatures}\label{S: Wall}

This section contains our non-additivity theorems.  We prove our first main result, on
non-additivity of perverse signatures
for pseudomanifolds, in the first subsection, then obtain our
second main result, for $\bd$-stratified pseudomanifolds, as a corollary
in the second subsection.

\subsection{Non-additivity of perverse signatures for pseudomanifolds}

In this section, we prove a generalization of the Wall non-additivity theorem for the perverse pairings of intersection homology theory. The general outline of the proof is the same as that in \cite{Wa69}, but there are some subtleties and generalizations that need to be addressed.

Throughout this section, let $X$ be a $\Q$-oriented s-closed stratified $4n$-pseudomanifold (it may possess codimension one strata). Let $Z\subset X$ be a bicollared codimension one subpseudomanifold such that $X=Y_1\cup_Z Y_2$ and $\bd Y_1=Z=-\bd Y_2$, accounting for orientations.

Let $V=I^{\bar q/\bar p}H_{2n}(Z;\Q)$ equipped with the anti-symmetric pairing $\Phi$ defined in Section \ref{S: pairing}. Let $A=\ker(i_{Z\subset Y_1}:I^{\bar q/\bar p}H_{2n}(Z;\Q)\to I^{\bar q/\bar p}H_{2n}(Y_1;\Q))$, $C=\ker(i_{Z\subset Y_2}:I^{\bar q/\bar p}H_{2n}(Z;\Q)\to I^{\bar q/\bar p}H_{2n}(Y_2;\Q))$, and  let $B=\ker(d: I^{\bar q/\bar p}H_{2n}(Z;\Q)\to I^{\bar p}H_{2n-1}(Z;\Q))$. 

\begin{theorem}\label{T: Wall}
With the assumptions and definitions considered above, 
$$\sigma_{\bar p\to\bar q}(X)=\sigma_{\bar p\onto\bar q}(Y_1)+\sigma_{\bar p\onto\bar q}(Y_2)+ \sigma(V;A,B,C),$$ 
where $\sigma$ is Wall's Maslov triple index. 
\end{theorem}

\begin{proof}[Proof of Theorem \ref{T: Wall}]
First we show that $I^{\bar p\to\bar q}H_{2n}(X;\Q)$ decomposes as a direct sum of 
$I^{\bar p\onto\bar q}H_{2n}(Y_1,Z;\Q)$, $I^{\bar p\onto\bar q}H_{2n}(Y_2,Z;\Q)$, and a third piece,
and that the signature pairing is diagonal with respect to this decomposition.
Consider   the morphisms (induced by inclusion)
\begin{align*}
I^{\bar p}H_{2n}(Y_1;\Q)\oplus I^{\bar p}H_{2n}(Y_2;\Q)&\to I^{\bar p}H_{2n}(X;\Q) \\
&\onto  I^{\bar p\to\bar q}H_{2n}(X;\Q)\\
&\into  I^{\bar q}H_{2n}(X;\Q)\\
&\to  I^{\bar q}H_{2n}(X,Z;\Q)\\
& \cong I^{\bar q}H_{2n}(Y_1,Z ;\Q)\oplus I^{\bar q}H_{2n}(Y_2,Z ;\Q).
\end{align*}
The image of the composition is $I^{\bar p\onto \bar q}H_{2n}(Y_1,Z;\Q)\oplus I^{\bar p\onto \bar q}H_{2n}(Y_2,Z;\Q)$, so we have an induced surjection
{\footnotesize
$$
\im(I^{\bar p} H_{2n}(Y_1;\Q)\oplus I^{\bar p} H_{2n}(Y_2;\Q)\to I^{\bar p\to \bar q} H_{2n}(X;\Q))
\to I^{\bar p\onto \bar q}H_{2n}(Y_1,Z;\Q)\oplus I^{\bar p\onto \bar q}H_{2n}(Y_2,Z;\Q).
$$}
Since all groups are really $\Q$-vector spaces,  there is a  (non-unique) splitting of this map.
 We claim that this  splitting is isometric in that it preserves the intersection pairing (where the intersection pairing on $I^{\bar p\onto \bar q}H_{2n}(Y_1,Z;\Q)\oplus I^{\bar p\onto \bar q}H_{2n}(Y_2,Z;\Q)$ is given by the orthogonal sum). 
Indeed, suppose that $[x]=[x_1]+[x_2]$ and $[y]=[y_1]+[y_2]$ are elements of $I^{\bar p\onto\bar q}H_{2n}(Y_1,Z;\Q)\oplus I^{\bar p\onto\bar q}H_{2n}(Y_2,Z;\Q)$ and that $[\td x]=[\td x_1]+[\td x_2]$ and $[\td y]=[\td y_1]+[\td y_2] \in \im(I^{\bar p} H_{2n}(Y_1;\Q)\oplus I^{\bar p} H_{2n}(Y_2;\Q)\to I^{\bar p\to \bar q} H_{2n}(X;\Q))$ are their images under a given splitting. Each $[\td x_i]$ and $[\td y_i]$ may be represented by $\bar p$-allowable cycles with support in the interior of $Y_i$, and the same cycles represent the $[x_i]$ and $[y_i]$. Furthermore, we can always assume that the representing cycles are in stratified general position. Then, by definition, (the augmentations of) both intersection pairings are given by counting the intersection numbers of the the representative chains, and it is clear that chains in $Y_1$ do not intersect those in $Y_2$. So $x\pfa y= x_1\pfa y_1+x_2\pfa y_2= \td x_1\pfa \td y_1+\td x_2\pfa \td y_2=\td x\pfa \td y$, and the pairing is preserved. 
Notice that if we choose a different splitting that, say, takes $[x]$ to $[\td x']$, then $[\td x-\td x']$ must map to $0$ in $I^{\bar q}H_{2n}(X,Z;\Q)$, i.e. it is $\bar q$-homologous in $X$ to a chain in $Z$. Such a chain clearly does not intersect any $\bar p$-allowable chain in the interior of either $Y_i$, and this explains why the choice of splitting does not affect the isometry type of the pairing.

Now, continuing with the proof of Theorem \ref{T: Wall}, we fix a splitting, and by an abuse of 
notation, let $I^{\bar p\onto\bar q}H_{2n}(Y_i,Z;\Q)$ also denote its image under the splitting 
in $I^{\bar p\to\bar q}H_{2n}(X;\Q)$.
It is geometrically clear that  a chain from  $I^{\bar p\onto\bar q}H_{2n}(Y_1,Z;\Q)$ does not intersect a chain from $I^{\bar p\onto\bar q}H_{2n}(Y_2,Z;\Q)$ in $I^{\bar p\to \bar q}H_{2n}(X;\Q)$. Thus $I^{\bar p\onto\bar q}H_{2n}(Y_1,Z;\Q)$ and $I^{\bar p\onto\bar q}H_{2n}(Y_2,Z;\Q)$ are orthogonal in $I^{\bar p\to\bar q}H_{2n}(X;\Q)$. We also know that the restriction of the intersection pairing to each subspace is  nonsingular, and it follows that they must intersect trivially (e.g. any $[x]\in I^{\bar p\onto\bar q}H_{2n}(Y_1,Z;\Q)$ annihilates $I^{\bar p\onto\bar q}H_{2n}(Y_2,Z;\Q)$, so if also $[x]\in I^{\bar p\onto\bar q}H_{2n}(Y_2,Z;\Q)$, then $[x]$ must be $0$ or nonsingularity of the restriction of the form to $I^{\bar p\onto\bar q}H_{2n}(Y_2,Z;\Q)$ would be violated). Together, the subspace $J=I^{\bar p\onto\bar q}H_{2n}(Y_1,Z;\Q)\oplus I^{\bar p\onto\bar q}H_{2n}(Y_2,Z;\Q)$ is thus an orthogonal sum, i.e. $J=I^{\bar p\onto\bar q}H_{2n}(Y_1,Z;\Q)\perp I^{\bar p\onto\bar q}H_{2n}(Y_2,Z;\Q)$, and the restriction of the pairing $\pfc$ to this subspace of $I^{\bar p\to \bar q}H_{2n}(X;\Q)$ is also nonsingular. 

Let $K$ be the annihilator of $J$ in $I^{\bar p\to \bar q}H_{2n}(X;\Q)$, i.e. $K=J^{\perp}$. Once again, the nonsingularity of the pairing on $J$ implies $J\cap K=0$, and in fact, it follows from basic linear algebra (see, e.g., \cite[Theorem 3.1]{MH}) that $I^{\bar p\to \bar q}H_{2n}(X;\Q)\cong J\perp J^{\perp}=J\perp K$.
Thus, choosing an appropriate basis, we can write $I^{\bar p\to\bar q}H_{2n}(X;\Q)$ as a direct sum with respect to which the intersection pairing is block diagonal, as desired: 
\begin{equation}\label{E: decomp}
I^{\bar p\to\bar q}H_{2n}(X;\Q)\cong I^{\bar p\onto\bar q}H_{2n}(Y_1,Z;\Q)\perp I^{\bar p\onto\bar q}H_{2n}(Y_2,Z;\Q)\perp K.
\end{equation}

It follows that $$\sigma_{\bar p\to\bar q}(X)= \sigma_{\bar p\onto\bar q}(Y_1)+\sigma_{\bar p\onto\bar q}(Y_2)+\sigma(K),
$$
so we must show that $\sigma(K)=\sigma(V;A,B,C)$. To do this, we will first decompose $K$ as a direct sum
$L \oplus S \oplus M$ and show that $\sigma(K) = \sigma(L)$.  Then we will show that we can
identify $\sigma(L)$ with the desired Maslov triple index.

Let $$S=K\cap \im(I^{\bar p}H_{2n}(Y_1;\Q)\oplus I^{\bar p}H_{2n}(Y_2;\Q)\to I^{\bar p\to \bar q}H_{2n}(X;\Q)),$$ and let $S^{\perp}$ denote the annihilator of $S$ in $K$ under $\pfc$.

\begin{remark}\label{R: kill} Note that $S^{\perp}$ is therefore also the annihilator of $$I^{\bar p\onto\bar q}H_{2n}(Y_1,Z;\Q)\perp I^{\bar p\onto\bar q}H_{2n}(Y_2,Z;\Q)\perp S$$ in $I^{\bar p\to\bar q}H_{2n}(X;\Q)$.
\end{remark}

\begin{lemma}
Under the nonsingular pairing $\pfb: I^{\bar p}H_{2n}(X;\Q)\otimes I^{\bar q}H_{2n}(X;\Q)\to \Q$,
the annihilator of  $$\im(I^{\bar p}H_{2n}(Y_1;\Q)\oplus I^{\bar p}H_{2n}(Y_1;\Q)\to I^{\bar p}H_{2n}(X;\Q))$$  is $$\im(I^{\bar q}H_{2n}(Z;\Q)\to I^{\bar q}H_{2n}(X;\Q)).$$
\end{lemma}
\begin{proof}
The orthogonality of the two spaces is geometrically evident as $\bar p$-allowable chains in $Y_1$ or $Y_2$ can be assumed to have support in the interior of the space (i.e. away from the stratified boundary $Z$), due to the bicollar on $Z$.

On the other hand, suppose $[x]$ is in $I^{\bar q}H_{2n}(X;\Q)$ but not in $\im(I^{\bar q}H_{2n}(Z;\Q)\to I^{\bar q}H_{2n}(X;\Q))$. Then, by the long exact sequence of the pair, $[x]$ has a nontrivial image in $I^{\bar q}H_{2n}(X,Z;\Q)\cong I^{\bar q}H_{2n}(Y_1,Z;\Q)\oplus I^{\bar q}H_{2n}(Y_2,Z;\Q)$. But since $I^{\bar q}H_{2n}(Y_i,Z;\Q)$ is dual to $I^{\bar p}H_{2n}(Y_i;\Q)$, this implies there must be some $[y]\in I^{\bar p}H_{2n}(Y_i;\Q)$ for $i=0$ or $i=1$ such that $y$ and $x$ intersect with non-zero intersection number. Therefore $[x]$ cannot be orthogonal to $\im(I^{\bar p}H_{2n}(Y_1;\Q)\oplus I^{\bar p}H_{2n}(Y_1;\Q)\to I^{\bar p}H_{2n}(X;\Q))$. 

Thus the annihilator of  $\im(I^{\bar p}H_{2n}(Y_1;\Q)\oplus I^{\bar p}H_{2n}(Y_1;\Q)\to I^{\bar p}H_{2n}(X;\Q))$ must be exactly $\im(I^{\bar q}H_{2n}(Z;\Q)\to I^{\bar q}H_{2n}(X;\Q))$.
\end{proof}

\begin{corollary}\label{C: S perp}
$S^{\perp}=I^{\bar p\to \bar q}H_{2n}(X;\Q)\cap \im(I^{\bar q}H_{2n}(Z;\Q)\to I^{\bar q}H_{2n}(X;\Q))$.
\end{corollary}
\begin{proof}
First, observe that any $[x]\in I^{\bar p\to \bar q}H_{2n}(X;\Q)\cap \im(I^{\bar q}H_{2n}(Z;\Q)\to I^{\bar q}H_{2n}(X;\Q))$ is in $K$ because any cycle that has a $\bar q$-allowable representative with support in $Z$ must have trivial intersection with any $\bar p$-allowable chain with support in the interior of $Y_1$ or $Y_2$, and any chain in either $I^{\bar p\onto\bar q}H_{2n}(Y_i,Z;\Q)$ can be so represented. For the same reason, $[x]\in S^{\perp}$ because any element of $S$ can be so represented. Thus 
$$I^{\bar p\to \bar q}H_{2n}(X;\Q)\cap \im(I^{\bar q}H_{2n}(Z;\Q)\to I^{\bar q}H_{2n}(X;\Q))\subset S^{\perp}.$$

On the other hand, anything in $S^{\perp}$ lies in $K\subset I^{\bar p\to \bar q}H_{2n}(X;\Q)$ by definition, and so by the preceding lemma, to complete the proof we need only show that $S^{\perp}$ annihilates $\im(I^{\bar p}H_{2n}(Y_1;\Q)\oplus I^{\bar p}H_{2n}(Y_1;\Q)\to I^{\bar p}H_{2n}(X;\Q))$ under the pairing $\pfb: I^{\bar p}H_{2n}(X;\Q)\otimes I^{\bar q}H_{2n}(X;\Q)\to \Q$. 
Suppose $[x]\in S^{\perp}\subset I^{\bar q}H_{2n}(X;\Q_0)$,  $[y]\in  \im(I^{\bar p}H_{2n}(Y_1;\Q)\oplus I^{\bar p}H_{2n}(Y_2;\Q)\to I^{\bar p}H_{2n}(X;\Q))$, and $[y]\pfb [x]\neq 0$. Since  $[x]\in  I^{\bar p\to \bar q}H_{2n}(X;\Q)$, we can choose a $\bar p$-allowable representative for $[x]$, and we can think of $y$ as $\bar q$-allowable, obtaining the well-defined equalities $[x]\pfb[y]=\epsilon[x\pfa y]=\pm\epsilon[y\pfa x]=\pm[y]\pfb[x]\neq 0$.  But now in the expression $[x]\pfb[y]$,  the cycle $[y]$ is representing an element of $\im(I^{\bar p}H_{2n}(Y_1;\Q)\oplus I^{\bar p}H_{2n}(Y_1;\Q)\to I^{\bar p\to \bar q}H_{2n}(X;\Q))$, which we claim is a subset of $I^{\bar p\onto\bar q}H_{2n}(Y_1,Z;\Q)\perp I^{\bar p\onto\bar q}H_{2n}(Y_2,Z;\Q)\perp S$. This would  contradict the 
assumption that $[x]\in S^{\perp}$ by Remark \ref{R: kill}, and so $S^{\perp}$ would indeed annihilate $\im(I^{\bar p}H_{2n}(Y_1;\Q)\oplus I^{\bar p}H_{2n}(Y_1;\Q)\to I^{\bar p}H_{2n}(X;\Q))$.

To verify the claim, suppose $[y]\in im(I^{\bar p}H_{2n}(Y_1;\Q)\oplus I^{\bar p}H_{2n}(Y_1;\Q)\to I^{\bar p\to \bar q}H_{2n}(X;\Q))$. We can write $[y]=[y_1+y_2]$, where $y_i$ is a $\bar p$-allowable cycle in $Y_i$. Furthermore, using the decomposition formula \eqref{E: decomp}, we can write $[y]=[z_1+z_2+z_3]$, where $z_i$ is a $\bar p$-allowable cycle in  $Y_i$ and $[z_3]\in K$. But then $[z_3]=[y_1-z_1+y_2-z_2]$, and since each $y_i-z_i$ is a $\bar p$-allowable cycle in $Y_i$, it follows that $[z_3]\in S$ by the definition of $S$. Thus the decomposition  $[y]=[z_1]+[z_2]+[z_3]$ provides the desired conclusion. 
\end{proof}

\begin{lemma}
$S\subset S^{\perp}$.
\end{lemma}
\begin{proof}
$S\subset K\subset I^{\bar p\to\bar q}H_{2n}(X;\Q)$ by definition, so it suffices, by the preceding corollary, to show that $S\subset \im(I^{\bar q}H_{2n}(Z;\Q)\to I^{\bar q}H_{2n}(X;\Q))$.

Recall that $S$ is the intersection of the image of $I^{\bar p}H_{2n}(Y_1;\Q)\oplus I^{\bar p}H_{2n}(Y_2;\Q)$ in $I^{\bar p\to \bar q}H_{2n}(X;\Q)$ with $K$, which is the annihilator and additive complement of $I^{\bar p\onto \bar q}H_{2n}(Y_1,Z;\Q)\oplus I^{\bar p\onto \bar q}H_{2n}(Y_2,Z;\Q)$. Thus if $[x]\in S$, then $x$ can be written as the sum of two $\bar p$-allowable cycles, one each in $Y_1$ and $Y_2$. These cycles have well-defined images respectively in $I^{\bar p\onto\bar q}H_{2n}(Y_1,Z;\Q)\subset I^{\bar q}H_{2n}(Y_1,Z;\Q)\cong  I^{\bar q}H_{2n}(X,Y_2;\Q)$ and $I^{\bar p\onto\bar q}H_{2n}(Y_2,Z;\Q)\subset I^{\bar q}H_{2n}(Y_2,Z;\Q)\cong  I^{\bar q}H_{2n}(X,Y_1;\Q)$. But each of these images must be $0$ since the pairing on each $I^{\bar p\onto\bar q}H_{2n}(Y_i,Z;\Q)$
is nonsingular and $[x]$ is orthogonal to everything in these spaces. Thus $[x]$ must be $0$ in $I^{\bar q}H_{2n}(Y_1,Z;\Q)\oplus I^{\bar q}H_{2n}(Y_2,Z;\Q)\cong I^{\bar q}H_{2n}(X,Z;\Q)$. Therefore
by the relative sequence, $[x]$ must be in the image of $I^{\bar q}H_{2n}(Z;\Q)$. 
\end{proof}

We can now proceed as in Wall \cite{Wa69}:

Since the intersection form restricted to $K$ is nonsingular, we have\footnote{Clearly $S\subset (S^{\perp})^{\perp}$, and then we can apply dimension counting ($\dim(S^{\perp})=\dim(K)-\dim(S)$) since $K$ is finite-dimensional.} $(S^{\perp})^{\perp}=S$. Therefore the radical of the restriction of the intersection form to $S^{\perp}$ is $S$. This implies that the form is nonsingular when restricted to any additive complement
$L$ of $S$ in $S^{\perp}$. We can thus complete the multiplication table for the intersection pairing on $K$ as follows:

\begin{center}
\begin{tabular}{r|ccc}
&$L$&$S$&$M$\\ \hline
$L$&$A$&$0$&$0$\\
$S$&$0$&$0$&$B$  \\
$M$&$0$&$B^t$&$D$
\end{tabular}
\end{center} 

Here $L\oplus S=S^{\perp}$, $M$ is an additive complement of $S^{\perp}$, chosen so that  $L$ and $M$ are mutually annihilating under $\pfc$. This can be done by an appropriate change of basis if necessary, because the pairing is nonsingular on $L$. The letters $A,B,D$ represent matrices of intersection numbers, and $B^t$ is the transpose of $B$. It appears because the intersection pairing is symmetric. Furthermore, the form on $S\oplus M$ must be nonsingular, because the entire pairing on $K$ is nonsingular, and the annihilator of $S$ in $S\oplus M$ must be $(S\oplus M)\cap S^{\perp}=S$. So $S$ is self-annihilating on $S\oplus M$, which implies that the signature of the pairing on $S\oplus M$ is\footnote{It is a standard fact about nonsingular bilinear forms that their signatures are $0$ if they possess a subspace $U$ such that $U=U^{\perp}$, but it is harder than expected to find a clear, concise proof in the expository topology literature. Thus we include a brief proof here, owing largely to the treatment in \cite{BW1}.

Suppose we have a nonsingular symmetric form $(\cdot,\cdot)$ on the finite dimensional vector space $V$. Let $V_+, V_-$ be the maximal positive definite, respectively negative definite, subspaces of $V$. Then by definition, the signature of the form is $\sigma=\dim(V_+)-\dim (V_-)$. Let $U$ be a subspace such that $U^{\perp}=U$. Since $\dim(U)+\dim(U^\perp)=\dim(V)$, $\dim(U)=\frac{1}{2}\dim(V)$. 
Clearly also $U\cap V_+=U\cap V_-=0$. But then $\dim(V)\geq \dim(U)+\dim(V_+)-\dim(U\cap V_+)=\dim(U)+\dim(V_+)$ and $\dim(V)\geq \dim(U)+\dim(V_-)-\dim(U\cap V_-)=\dim(U)+\dim(V_-)$ by the inclusion/exclusion formula. Thus $\frac{1}{2}\dim(V)\geq \dim(V_+)$ and $\frac{1}{2}\dim(V)\geq \dim(V_-)$. But by diagonalizing the form, it follows easily from non-singularity that $\dim(V_+)+\dim(V_-)=\dim(V)$. This forces $\dim(V_+)=\dim(V_-)=\frac{1}{2}\dim(V)$. Thus the signature must be $0$.} $0$. 
It readily follows that $\sigma(K)=\sigma(L)$. 

Finally, we want to identify $\sigma(L)$ with the indicated Maslov triple index.  To do this, we first
will identify $L$ with a space $W$ defined using the spaces $A$, $B$, and $C$ that occur in the 
Maslov index.
For this part of the proof, will refer to the 
following commutative diagram of long exact sequences derived from the diagram \eqref{E: sequences}.
\begin{diagram}
&\rTo &I^{\bar p}H_{2n+1}(X,Z; \Q) &\rTo^{\delta}& I^{\bar p}H_{2n}(Z; \Q)&\rTo^{(i_{Z \subset X})_*}& I^{\bar p}H_{2n}(X; \Q)&\rTo &\\
&&\dTo^{(i_{\bar p\to\bar q})_*}&&\dTo^{(i_{\bar p\to\bar q})_*}&&\dTo^{(i_{\bar p\to\bar q})_*}&\\
&\rTo &I^{\bar q}H_{2n+1}(X,Z; \Q) &\rTo^{\delta}& I^{\bar q}H_{2n}(Z; \Q)&\rTo^{(i_{Z \subset X})_*}& I^{\bar q}H_{2n}(X; \Q)&\rTo &\\
&&\dTo^{(\pi_{\bar q/\bar p})_*}&&\dTo^{(\pi_{\bar q/\bar p})_*}&&\dTo^{(\pi_{\bar q/\bar p})_*}&\\
&\rTo &I^{\bar q/\bar p}H_{2n+1}(X,Z; \Q) &\rTo^{\delta}& I^{\bar q/\bar p}H_{2n}(Z; \Q)&\rTo^{(i_{Z \subset X})_*}& I^{\bar q/\bar p}H_{2n}(X; \Q)&\rTo & \\
&&\dTo^{d}&&\dTo^{d}&&\dTo^{d}&\\
&\rTo &I^{\bar p}H_{2n}(X,Z; \Q) &\rTo^{\delta}& I^{\bar p}H_{2n-1}(Z; \Q)&\rTo^{(i_{Z \subset X})_*}& I^{\bar p}H_{2n-1}(X; \Q)&\rTo&.
\end{diagram}
Recall again that the ``connecting map'' $d$ can  be described as acting on a representative chain $x$ by taking $x$ to its boundary $\bd x$, that  an element of $I^{\bar q/\bar p}H_*(X,Y;G))$ can be represented by a chain $x$ such that $\bd x=a+b$ with $b$ a $\bar q$-allowable chain in $Y$ and $a$ a $\bar p$-allowable chain in $X$, and, in this case, $\delta x$ is represented by $b$.

As above,
let 
\begin{equation}
A=\ker((i_{Z \subset Y_1})_*:I^{\bar q/\bar p}H_{2n}(Z;\Q)\to I^{\bar q/\bar p}H_{2n}(Y_1;\Q)),
\end{equation}
\begin{equation}
C=\ker((i_{Z \subset Y_2})_*:I^{\bar q/\bar p}H_{2n}(Z;\Q)\to I^{\bar q/\bar p}H_{2n}(Y_2;\Q)),\end{equation}
and
\begin{equation}
B=\ker(d: I^{\bar q/\bar p}H_{2n}(Z;\Q)\to I^{\bar p}H_{2n-1}(Z;\Q)),
\end{equation}
and define
\begin{equation}
W=\frac{B\cap (C+A)}{B\cap C+B\cap A}.
\end{equation}
The reader may want to refer back to Section \ref{S: Wall1} to recall how $W$ comes into Wall's Maslov index.

We seek to define a map $f: S^{\perp}\to W$. To begin, 
 for $[z_1] \in S^\perp$, define an element $\tilde{f}([z_1])\in I^{\bar q/\bar p}H_{2n}(Z;\Q)$ as follows.
If $[z_1] \in S^{\perp} \cong I^{\bar p \to \bar q}H_{2n}(X;\Q) \cap \im(I^{\bar q}H_{2n}(Z;\Q) \rightarrow I^{\bar q}H_{2n}(X;\Q))$, then $[z_1]$ is an element of $I^{\bar q}H_{2n}(X;\Q)$, and
we can lift $[z_1]$ (non-uniquely) to $[z_2]\in I^{\bar q}H_{2n}(Z;\Q)$. Let
$\tilde{f}([z_1]) = (\pi_{\bar q/ \bar p})_*([z_2])\in I^{\bar q/\bar p}H_{2n}(Z;\Q)$. 

\begin{proposition} For every $[z_1]\in S^\perp$, 
$\tilde{f}([z_1]) \in B\cap (C+A)$.  Further, up to elements of $B\cap A + B\cap C$,
$\tilde{f}([z_1])$ is independent of 
the choice of $[z_2]$ made in the definition, 
so $\tilde{f}$ defines a map $f:S^\perp \rightarrow W$.  
The map $f$ is a homomorphism.
\end{proposition}

\begin{proof}
We begin by demonstrating that  $\tilde{f}([z_1])\in B\cap (C+A)$. 
First, $\tilde{f}([z_1])$ is by construction in $\im((\pi_{\bar q/ \bar p})_*)$, so by exactness of the $\bar q/\bar p$ sequence, it lies in $B$. 
To see that $\tilde{f}([z_1])$ also lies in $A+C$, we note that 
$$
A\cong \im(\delta:I^{\bar q/\bar p}H_{2n+1}(Y_1,Z;\Q)\to I^{\bar q/\bar p}H_{2n}(Z;\Q))
$$
 and 
 $$
 C\cong \im(\delta:I^{\bar q/\bar p}H_{2n+1}(Y_2,Z;\Q)\to I^{\bar q/\bar p}H_{2n}(Z;\Q))
 $$ 
 by the long exact sequence \eqref{E: LES}. So
\begin{multline}A+C=\im(\delta+\delta: I^{\bar q/\bar p}H_{2n+1}(Y_1,Z;\Q)\oplus I^{\bar q/\bar p}H_{2n+1}(Y_2,Z;\Q)\to I^{\bar q/\bar p}H_{2n}(Z;\Q))\\
\cong \im(\delta: I^{\bar q/\bar p}H_{2n+1}(X,Z;\Q)\to I^{\bar q/\bar p}H_{2n}(Z;\Q)).
\end{multline}
Now let us go through the definition of $\tilde{f}$ again carefully, referring to the diagram above.  
If $[z_1]\in S^{\perp}$, then by Corollary \ref{C: S perp}, $z_1$ can be represented by a $\bar p$-allowable cycle, $a$, in $X$ (representing a class in $I^{\bar p}H_{2n}(X;\Q)$) that is $\bar q$-homologous to a $\bar q$-allowable cycle, $b$, in $Z$ which
represents the class $[z_2]$ in the definition of $\tilde{f}$. So $(i_{\bar q/ \bar p})_*([a]) =
(i_{Z \subset X})_*([b])$ in $I^{\bar q}H_{2n}(X;\Q)$.
Let $\xi$ be a $\bar q$-allowable $2n+1$ chain in $X$ realizing such a homology, i.e. $\bd \xi=b-a$. Since $a$ is $\bar p$-allowable in $X$ and $b$ is $\bar q$-allowable in $Z$, $\xi$ represents  an element of  $I^{\bar q/\bar p}H_{2n+1}(X,Z;\Q)$.  By definition of $\delta$, the  image $\delta(\xi)$ in $I^{\bar q/\bar p}H_{2n}(Z;\Q)$ is also represented by $b$, and so is $(\pi_{\bar q/\bar p})_*[b]=\tilde{f}([z_1])$. Thus $\tilde{f}([z_1])\in \im(\delta) = A+C$.

Next, suppose that $[z_2],[z_2']\in I^{\bar q}H_{2n}(Z;\Q)$ are two choices of lifts for the same $[z_1]\in S^{\perp}\subset \im(I^{\bar q}H_{2n}(Z;\Q)\to I^{\bar q}H_{2n}(X;\Q))$, i.e., $(i_{Z\subset X})_*([z_2])=(i_{Z\subset X})_*([z_2'])=[z_1]$.  Let $\tilde{f}([z_1]):=(\pi_{\bar q/\bar p})([z_2])$ and $\tilde{f}'([z_1]):=(\pi_{\bar q/\bar p})([z'_2])$.
We need to show that 
$\tilde{f}([z_1]) - \tilde{f}'([z_1]) \in B\cap A + B\cap C$.

We have 
\begin{align*}
[z_2]-[z_2']&\in \im(\delta:I^{\bar q}H_{2n+1}(X,Z;\Q)\to I^{\bar q}H_{2n}(Z;\Q))\\
&\cong \im(\delta + \delta :I^{\bar q}H_{2n+1}(Y_1,Z;\Q)\oplus I^{\bar q}H_{2n+1}(Y_2,Z;\Q)\to I^{\bar q}H_{2n}(Z;\Q)).
\end{align*}
Let $[x_1], [x_2]$ be elements, respectively, of $I^{\bar q}H_{2n+1}(Y_1,Z;\Q)$, $I^{\bar q}H_{2n+1}(Y_2,Z;\Q)$ such that 
$\delta [x_1]+ \delta [x_2]=[z_2]-[z_2']$. Then $\delta \circ (\pi_{\bar q/\bar p})_* ([x_1]) \in B\cap A$ and 
$\delta \circ (\pi_{\bar q/\bar p})_* ([x_2])\in B\cap C$. So 
$(\pi_{\bar q/\bar p})([z_2] - [z_2']) =(\pi_{\bar q/\bar p}) \circ \delta ([x_1] +[x_2]) \in B\cap A + B\cap C$
as required.

Putting our arguments so far together, we see that $f$ is well-defined as a function from $S^{\perp}$ to $W$. But $f$ is then also a homomorphism since for a sum $[z_1]+[z_1']$, we can certainly find a lift of the form $[z_2]+[z_2']$, we have just shown that this choice is acceptable and does not affect the image, and the other maps are all homomorphisms. 
\end{proof}

Now we can let $f([z_1]) = [z_3]$, where by an abuse of notation, $[z_3]$ is also taken
to be the class that $z_3$ represents in $W$ as $\tilde{f}([z_1])$.

\begin{proposition}\label{P: L=W}
The map $f$ surjects onto $W$ with kernel
$S$, hence $L\cong W$.
\end{proposition}
\begin{figure}[!htp]
\begin{center}
\scalebox{.4}{\includegraphics{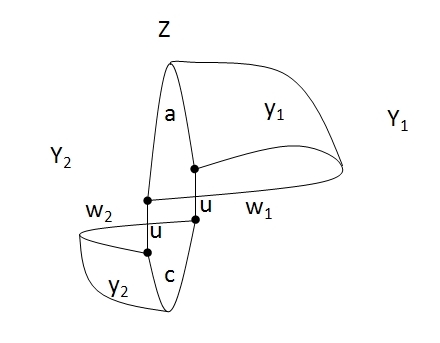}}
\caption{A schematic for the argument that $f$ is surjective.}\label{F: fig1}
\end{center}
\end{figure}
\begin{proof}
First we show that $f:S^{\perp}\to W$ is surjective. Suppose $[x]$ is a class in $B\cap (A+C)$. 
Since $[x] \in B$, there exists an $[x_2] \in I^{\bar q}H_{2n}(Z; \Q)$ with
$(\pi_{\bar q/\bar p})_*[x_2]=[x]$.  To get surjectivity of $f$, we need to show there is a 
$[v] \in I^{\bar p}H_{2n}(X;\Q)$ such that $(i_{\bar p\to\bar q})_*[v] = (i_{Z \subset X})_*[x_2]:=[x_1]$.
Then $f[x_1]=[x]$.

It might aid the reader to refer to the schematic in Diagram \eqref{F: fig1} during the following argument.

Since $[x] \in A +C$, we can write $[x]=[a]+[c]$, where $[a]\in A$ and $[c]\in C$. Since $[a]\in A$, it is
the image under $\delta$ of some $[y_1]$, which may be represented by a $\bar q$-allowable chain $y_1$ with support in $Y_1$ such that $\bd y_1=a+w_1$ and $w_1$ is a $\bar p$-allowable relative
$2n$ chain in $Y_1$. Similarly, there is a $\bar q$-allowable chain $y_2$ with support in $Y_2$ such that $\bd y_2=c+w_2$ and $w_2$ is a $\bar p$-allowable $2n$ chain in $Y_2$. 
Then $d[y_1]=[w_1]$ and $d[y_2]=[w_2]$.
Since $x$ is also in $B$, $d[x] = [\bd x] =[0] \in I^{\bar p}H_{2n-1}(Z;\Q)$, so
there is a $\bar p$-allowable chain $u$ in $Z$ such that $\bd u=\bd x=\bd a+\bd c$. 
Now consider the $\bar p$-allowable $2n$-chain $v=w_1+u+w_2$. We have 
\begin{align*}
\bd (w_1+u+w_2)&= \bd w_1+\bd u+\bd w_2\\
&=-\bd a + \bd a+\bd c -\bd c\\
&=0,
\end{align*}
so $v$ represents a class $[v] \in I^{\bar p}H_{2n}(X;\Q)$, and 
\begin{align*}
\bd (y_1+y_2) &= \bd y_1+\bd y_2\\
&=a+w_1+c+w_2\\
&=a+c-u +w_1+u+w_2.
\end{align*}
In other words, $y_1+y_2$ provides a $\bar q$-allowable homology from the $\bar p$-allowable cycle $-(w_1+u+w_2)$ in $X$ to the $\bar q$-allowable cycle $a+c-u$ in $Z$.  This means that
$(i_{\bar p\to\bar q})_*[v] = (i_{Z \subset X})_*[a+c-u]$, so we can set $[x_2] = [a+c-u]$
and we get that $[x_1] = (i_{Z \subset X})_*[a+c-u] \in S^\perp$ by Corollary \ref{C: S perp}.
Finally, since $u$ is $\bar p$-allowable, we have that $(\pi_{\bar q/\bar p})_*[x_2]=(\pi_{\bar q/\bar p})_*[a+c-u]=[a+c]=[x]$ as desired.

Lastly, we must show that $\ker f=S$, which will suffice, as $L$ is an additive complement of $S$ in $S^{\perp}$. 
Suppose $[x]\in S= K\cap \im(I^{\bar p}H_{2n}(Y_1;\Q)\oplus I^{\bar p}H_{2n}(Y_2;\Q)\to I^{\bar p\to \bar q}H_{2n}(X;\Q))$.  Then we can write $[x]=(i_{\bar p\to\bar q})_*[x_1+x_2]$, where each $x_i$  is a $\bar p$-allowable cycle with support in $Y_i$. Since $[x]\in K$, we get $(\pi_{X,Z})_*[x] =[0] \in I^{\bar q}H_{2n}(Y_1,Z;\Q)\oplus I^{\bar q}H_{2n}(Y_2,Z;\Q)\cong I^{\bar q}H_{2n}(X,Z;\Q)$ (or else its image would be nontrivial in $I^{\bar p\onto \bar q}H_{2n}(Y_1,Z;\Q)\oplus I^{\bar p\onto \bar q}H_{2n}(Y_2,Z;\Q)$ and thus it could not be orthogonal to this group, on which the intersection pairing is nonsingular). So 
the cycle $x_1+x_2$ which we can take to represent $[x]$ is $\bar q$-homologous to a cycle in $Z$. 
But using that $I^{\bar q}H_{2n}(X,Z;\Q)\cong I^{\bar q}H_{2n}(Y_1,Z;\Q)\oplus I^{\bar q}H_{2n}(Y_2,Z;\Q)$, it follows that each of $x_1$ and $x_2$ is individually $\bar q$-homologous to a $\bar q$-allowable cycle in $Z$. Thus $[x_1]$ and $[x_2]$ are each individually elements of $K$ and so are individually elements of $S$. We claim that $f([x_1])\in B\cap A$ and $f([x_2])\in B\cap C$; the proofs are the same so we will just show the first. Let $y_1$ be a $\bar q$-homology in $Y_1$ from $x_1$ to a $\bar q$-chain $x_1'$ with support in $Z$. Then $x_1'$ represents $f([x_1])$, and it is clear that $f([x_1])\in B$ since $x_1'$ is a cycle. But it is also clear that $[x_1']\in A$, since $\bd y_1= x_1'-x_1$, which is the sum of a $\bar q$-allowable
chain on $Z$ and a $\bar p$-allowable chain on $Y_1$, thus $y_1$ is a cycle in 
$I^{\bar q/\bar p}H_{2n+1}(X,Z;\Q)$ and $\delta[y_1]= [x_1']$.
It follows now that $f(S)=0\in W$.

Conversely, suppose $[x]\in S^{\perp}$  and that $f([x])=0\in W$, i.e. $f([x])\in B\cap A+B\cap C$. We will show that $[x]\in S$; the reader may want to refer to Diagram \eqref{F: fig} for a schematic of the construction.
Since $[x]\in S^{\perp}$, we can assume by Corollary \ref{C: S perp} that $[x]$ is represented by a $\bar q$-allowable cycle $x$ supported in $Z$, and this same chain represents $f([x])$. Since $f([x])\in B\cap A+B\cap C\subset I^{\bar q/\bar p}H_{2n}(Z;\Q)$, there is a $\bar q$-allowable chain $z$ supported in $Z$ such that $\bd z=x-(x_1+x_2)-u$, where $[x_1]\in B\cap A$, $[x_2]\in B\cap C$, and $u$ is $\bar p$ allowable in $Z$. Note that $x_1$ and $x_1+u$ represent the same element of $B\cap A$, so we can represent $f([x])$ as $[x_1+u]+[x_2]$ in $B\cap A+B\cap C$.  
Since $[x_1+u]\in B$, $0=d[x_1+u]=[\bd(x_1+u)] \in I^{\bar p}H_{2n-1}(Z;\Q)$, so
there is a $2n$-dimensional $\bar p$-chain $w$ in $Z$ such that $\bd w=\bd (x_1+u)$, 
which are both $\bar p$-allowable. Notice also that $x_1+u-w$ is a cycle in the usual sense (its boundary is identically $0$), and $u-w$ is $\bar p$-allowable, so $[x_1+u-w]=[x_1] \in B\cap A$. 
Because this class is also in $A$, there is a $2n+1$ dimensional $\bar q$-chain $y_1$ in $Y_1$ such that $\bd y_1=x_1+u-w-p_1$, where $p_1$ is $\bar p$-allowable. Notice that $0=\bd \bd y_1=\bd (x_1+u-w)+\bd p_1=\bd p_1$, so $\bd p_1=0$, and $p_1$ is a cycle in $Y_1$. But now we observe that $[p_1]$ is in $S$: it is represented by a $\bar p$-cycle in $Y_1$, and since it is $\bar q$-homologous by $y_1$ to a cycle supported in $Z$, it is orthogonal to $I^{\bar p\onto \bar q}H_{2n}(Y_1,Z)\oplus I^{\bar p\onto \bar q}H_{2n}(Y_2,Z)$. 
Now observe that $\bd\bd z=0$ and $\bd x=0$, so that $\bd (x_1+u)=-\bd x_2$.  Thus by 
a similar argument, there is a $[p_2]\in S$ represented by a cycle $p_2$ in $Y_2$ that is $\bar q$-homologous by some $y_2$ to the cycle $x_2+w$. Putting these together, $p_1+p_2$ is $\bar q$-homologous to $x_1+u-w+x_2+w=x_1+u+x_2$. But we have already seen that $x_1+u+x_2$ is $\bar q$-homologous to $x$, and so $[p_1+p_2]=[x] \in I^{\bar q}H_{2n}(X;\Q)$. Thus $[x]\in S$.

\begin{figure}[!htp]
\begin{center}
\scalebox{.4}{\includegraphics{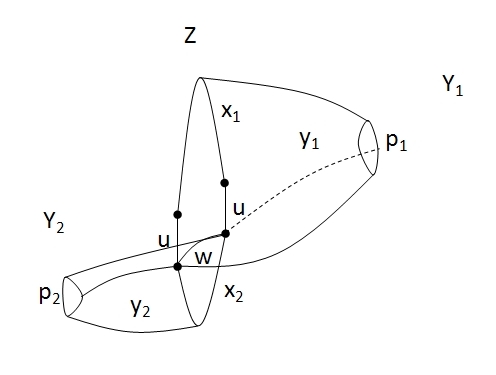}}
\caption{A schematic for the argument that $\ker(f)\subset S$.}\label{F: fig}
\end{center}
\end{figure}

\end{proof}

Finally, we must show that, under our isomorphism $L\cong W$, the signature of $L$ becomes the  Maslov index $\sigma(V;A,B,C)$ associated with the pairing $\Phi$ on $V=I^{\bar q/\bar p}H_{2n}(Z;\Q)$. For this Maslov triple index to make sense, we need the spaces $A$, $B$, and $C$ to be
self-annihilating  subspaces of $V$ under the pairing $\Phi([x],[y]):=\epsilon[ x \pfa \bd y + (-1)^{n-|x|}(\bd x) \pfa y]$, so
we need the following lemma.

\begin{lemma}
$\Phi(A\times A)=\Phi(B\times B)=\Phi(C\times C)=0$. 
\end{lemma}
\begin{proof}
It is clear that $\Phi(B\times B)=0$, for if $[x]\in B$, then $[x]\in im(I^{\bar q}H_{i}(Z;\Q) \to I^{\bar q/\bar p}H_{i}(Z;\Q))$. So $[x]$ can be represented as a $\bar q$-cycle, i.e. $\bd x=0$. Thus if $[x],[y]\in B$, certainly $\Phi([x],[y])=0$ from the definition. 

Now suppose $[x],[y] \in A$ are represented by $\bar q$-allowable chains $x,y$ in stratified general position in $Z$. 
The fact that $[x],[y]\in A$ means that there exist  $\bar q$-allowable $2n+1$-chains $\xi,\eta$ in 
$Y_1$ such that 
$\bd \xi=x+u$, $\bd \eta=y+v$, and $u,v$ are $\bar p$-allowable chains in $Y_1$. 
We can assume that $\xi$ and $\eta$ are in stratified general position rel $Z$ and  that in a collared neighborhood of $Z$, $\xi$ looks like $[0,1]\times x$ and $\eta$ looks like $[0,1]\times y$. Consider the chain $\xi\pfa v-u\pfa \eta$ in $Y_1$. Since $u$ and $v$ are $\bar p$-allowable and $\eta$ and $\xi$ are $\bar q$-allowable, this is a well-defined $\bar t$-allowable $1$-chain. Next we compute, using $\pfa_{Y_1}$ to denote intersection numbers in $Y_1$ and $\pfa_Z$ to denote those in $Z$:

\begin{align*}
\bd (\xi\pfa_{Y_1} v- u\pfa_{Y_1} \eta)& = (\bd \xi)\pfa_{Y_1} v+(-1)^{4n-|\xi|}\xi\pfa_{Y_1} \bd v  -(\bd u)\pfa_{Y_1} \eta -  (-1)^{4n-|u|} u\pfa_{Y_1} \bd \eta\\
&=(\bd \xi)\pfa_{Y_1} v-\xi\pfa_{Y_1} \bd v  -(\bd u)\pfa_{Y_1} \eta -  u\pfa_{Y_1} \bd \eta\\
&=  (x+u)\pfa_{Y_1} v+ \xi\pfa_{Y_1} \bd y   +\bd x\pfa_{Y_1} \eta -u \pfa_{Y_1} (y+v) \\
&=  x\pfa_{Y_1} v+u\pfa_{Y_1} v+ \xi\pfa_{Y_1} \bd y   +\bd x\pfa_{Y_1} \eta -u \pfa_{Y_1} y -u\pfa_{Y_1} v \\
&=   \xi\pfa_{Y_1} \bd y   +\bd x\pfa_{Y_1} \eta  \\
&=    x\pfa_Z \bd y +  (-1)^{|\bd x|} \bd x\pfa_Z y \\
&=   x\pfa_Z \bd y - \bd x\pfa_Z y \\
&=\hat \Phi([x],[y]).
\end{align*}

Here was have used that $Z$ is $4n-1$ dimensional, $Y_1$ is $4n$-dimensional, $x,y,u,v$ are $2n$-dimensional, and $\xi, \eta$ are $2n+1$ dimensional. We have also used the geometrically clear fact that $x$ does not intersect $v$ and $y$ does not intersect $u$, which follows from $x$ and $y$ being in stratified general position and our collar assumptions on $\xi$ and 
$\eta$. For the sign conventions relating intersection numbers in $Y_1$ with those in $Z$, see the Appendix. We conclude from this argument that the intersection number $\Phi([x],[y])$ must be $0$, as $\hat \Phi([x],[y])$ represents the boundary of a $1$-chain in $Y_1$. Thus $\Phi(A\times A)=0$.  An analogous argument shows that $\Phi(C\times C)=0$.
\end{proof}

Now we can relate the intersection pairing on $L \subset I^{\bar p\to \bar q}H_{2n}(X;\Q)$ to the pairing $\Phi$ on $V=I^{\bar q/\bar p}H_{2n}(Z;\Q)$. 
Suppose that $[x], [y] \in L \subset S^\perp$.  Then $[x]$ and $[y]$ can be represented by 
$\bar q$-allowable cycles  $x$ and $y$ in $Z$ that are homologous via $\bar q$-allowable chains
$\chi$ and $\gamma$ in $X$ to $\bar p$ allowable
cycles $\tilde x$ and $\tilde y$ in $X$.  By definition, 
$[x] \pfc [y]  = \epsilon[\tilde x \pfa_X y]$.

The representatives $x$ and $y$ descend also to represent 
classes $[x]$ and $[y]$
in $B \cap (A + C) \subset I^{\bar q/\bar p}H_{2n}(Z;\Q)$, which in turn represent
$f([x])$ and $f([y])$ in $W$.
Since $[y]\in A+C$, we can write $[y]= [a]+[c] \in I^{\bar q/\bar p}H_{2n}(Z;\Q)$, where $[a]\in A$, $[c]\in C$
are represented by $\bar q-$allowable chains in $Z$, and
$y$ is $\bar q$ homologous to $a+c+w$ for some $\bar p$-allowable chain $w$ on $Z$. 
Since $[a]\in A$ and $[c]=[c+w]\in C$,  there exist $\bar q$-allowable chains 
$\xi\in Y_1$ and $\eta\in Y_2$ such that $\bd \xi=a+u$, $\bd \eta=c+v+w$, and $u,v$ are $\bar p$-allowable chains in $Y_1$ and $Y_2$, respectively with 
boundaries in $Z$.  We can further assume that in the collar neighborhood of $Z$, $\xi$ and $\eta$ 
have a product structure $[-1,0]\times a$ and $[0,1] \times -v-w$  and that all chains are in stratified general position. 
Observe that $a+w+c$ is $\bar q$-homologous to the $\bar p$-cycle $-u-v$ via $\xi+\eta$.

Now again consider the pairing $[x] \pfc [y]$ in $I^{\bar p\to \bar q}H_{2n}(X;\Q)$.
We have $\tilde x \pfa_X y = \tilde x \pfa_X (a+c+w) = \tilde x \pfa_X (-u-v)$.  Now we have a 
$\bar p$ allowable chain on the right, so we can replace $\tilde x$ with the $\bar q$
allowable chain, $x \subset Z$ to which it is $\bar q$-homologous to obtain $x \pfa_X (-u-v)$.  
By pushing $x$ into $Y_1$ along the collar and using the product structure of $u$ near $Z$,
we get this is equal to the intersection
$x \pfa_{X} (-u)= x \pfa_{Z} (-\bd u)=x\pfa_{Z} \bd a$. But after augmentation, this is precisely the pairing $\Phi$ on 
$I^{\bar q/\bar p}H_{2n}(Z;\Q)$, $\Phi([x],[a]):=\epsilon[x\pfa_{Z} \bd a +(-1)^{n-|x|}(\bd x) \pfa_{Z} a]$ because $x$ is a cycle, and by definition this is in turn equal to $\Psi(f([x]),f([y]))$ on $W$.  Thus
the intersection pairing on $L$ may be identified with the pairing $\Psi$ on $W$ as desired.

To check the sign in this last equality, we must be careful about which roles $A$ and $C$ are playing. Certainly we have $\frac{B\cap(A+C)}{B\cap A+B\cap C}\cong \frac{B\cap(C+A)}{B\cap C+B\cap A}$ as spaces, but $A$ and $C$ play different roles in the pairing.  In Wall, the choice of which plays which role is determined so that $A$ is associated to the half of the space whose boundary orientation agrees with the orientation of the intersection and $C$ is associated with the space whose boundary has the opposite orientation of that assigned to the intersection. Thus we can let Wall's $A$ correspond to ours and Wall's $C$ corresponds to ours, and we can also use the order $B,C,A$ for these subspaces. So, since $[x]$ represents an element of $B$, our $\Phi([x],[a])$ corresponds to Wall's $$-\Phi(\text{element of first subspace},\text{element of last subspace}),$$ where the negative sign comes from our choice of $y=a+c$ rather than the  $y+a+c=0$ that Wall uses. So, using \eqref{E: Phi}, this is $\Psi([x],[y])$. 
Thus the intersection pairing restricted to $S^{\perp}$ is taken to Wall's pairing $\Psi$ determined from $\Phi$ on $W$, and we conclude by Wall's definition that $\sigma(L)=\sigma(W;A,B,C)$.

This completes the proof of the Theorem \ref{T: Wall}.
\end{proof}

\subsection{$\bd$-stratified Pseudomanifolds}\label{S: boundary}

If we start with an s-closed pseudomanifold $X$ and decompose it as $Y_1 \cup_Z Y_2$ along
a pseudomanifold $Z$ which is not a stratum of $X$, then $Y_1$ and $Y_2$ with the 
subspace stratification induced from $X$ are $\bd$-stratified pseudomanifolds.  We would
like to be able to further decompose $X$ by cutting the $Y_i$ into pieces, but to do 
this, we need a version of our non-additivity theorem for a $\bd$-stratified pseudomanifold.
To get this result as a corollary of Theorem \ref{T: Wall}, we use the restratification trick we discussed in Section 3.1.

For intuition, consider Figure \ref{F: fig3} below, where 
$X^{4k}$ is a compact oriented $\bd$-stratified pseudomanifold with  boundary $\bd X=W$ such that $X=Y_1\cup_Z Y_2$.  Assume that $Y_1,Y_2$ are compact oriented $\bd$-stratified pseudomanifolds  such that $Y_1\cap Y_2=\bd Y_1\cap\bd Y_2=Z$ is a bicollared (in $X$) $\bd$-stratified pseudomanifold  such that $\bd Y_1=Z\cup -W_1$, $\bd Y_2=W_2\cup -Z$, and $\bd W_1=\bd W_2=\bd Z$. Also assume that $(W_1,\bd W_1)\subset (Y_1,Z)$ and $(W_2,\bd W_2)\subset (Y_2,Z)$, $(Z,\bd Z)\subset (Y_1,W_1)$, and $(Z,\bd Z)\subset (Y_2,W_2)$ are collared as pairs. Note that $\bd X=W_2\cup - W_1$. The orientations are chosen to agree with Wall's conventions in \cite{Wa69} (see also Section \ref{S: manifold}, below).

\begin{figure}[!htp]
\begin{center}
\scalebox{.4}{\includegraphics{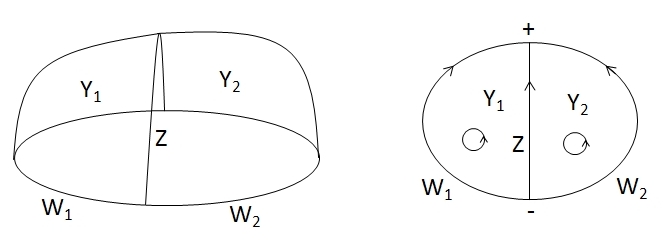}}
\caption{A schematic of a splitting of a $\bd$-stratified pseudomanifold  (left) and a flatter schematic of the relative orientations (right).}\label{F: fig3}
\end{center}
\end{figure}

Note that since $X$ is a $\bd$-stratified pseudomanifold with  boundary $W$, $W$ is \emph{not} a union of strata of $X$.
We begin by restratifying $X$ so that  $W$ becomes a union of strata, and we obtain  a stratified pseudomanifold  which we will denote 
by $\hat X$ (remember, though, that $X=\hat X$ as topological spaces). The strata of $\hat X$ are defined as follows:
\begin{enumerate}
\item for each stratum $S$ of $X$, then $S\cap X-W$ is a  stratum of $\hat X$,
\item for each stratum $S$ of $X$ such that $S\cap W\neq \emptyset$, then $S\cap W$ is a stratum of $\hat X$.
\end{enumerate}
\noindent
It is not hard to see that $\hat X$ is a PL stratified pseudomanifold. In fact, certainly $X$ and $\hat X$ agree off $W$, and if $N\cong W\times [0,1] \subset \hat X$ is a collared neighborhood of $W$ with $W=W\times \{1\}$, then under the subspace stratification from $\hat X$, $N$ is stratified as the product of $W$, with its stratification inherited from $X$, and $[0,1]$ with the stratification $[0,1]\supset \{1\}$. 
Let $\hat Y_i$, $\hat Z$ and $\hat W_i$ be the restratifications of $Y_i$, $Z$ and $W_i$ 
as subspaces of $\hat X$. Note that, with these stratifications, $\hat X$ and $\hat Z$ are PL stratified pseudomanifolds \emph{in particular without boundary}, while $\bd \hat Y_1=\hat Z$ and $\bd \hat Y_2=-\hat Z$.

Suppose $\bar p,\bar q$ are perversities on $X$, and induced also on the subspaces $Y_i$. 
Let $\hat p$ be the perversity on $\hat Y_i$ that agrees with $\bar p$ on $Y_i-W_i$ and is such that 
$\hat p(S)<0$ for all $S\subset \hat W_i$. Let $\bar q=\bar t-\bar p$. Note that then  $\hat q(S)>\bar t(S)$ for all $S\subset \hat W_i$.  Then we get the following isomorphisms of intersection homology groups.

\begin{lemma}
\begin{enumerate}
\item $I^{\hat p}H_*(\hat X;G)\cong I^{\bar p}H_*(X;G)$ and $I^{\hat q}H_*(\hat X;G)\cong I^{\bar q}H_*(X,\bd X; G)$,
\item  $I^{\hat p}H_*(\hat Z;G)\cong I^{\bar p}H_*(Z;G)$ and $I^{\hat q}H_*(\hat Z;G)\cong I^{\bar q}H_*(Z,\bd Z; G)$,
\item $I^{\hat p}H_*(\hat Y_i;G)\cong I^{\bar p}H_*(Y_i;G)$ and $I^{\hat q}H_*(\hat Y_i,\hat Z;G)\cong I^{\bar q}H_*(Y_i,\bd Y_i; G)$.\end{enumerate}
Therefore \begin{enumerate}
\item $I^{\hat p\to \hat q}H_*(\hat X;G)\cong I^{\bar p\onto \bar q}H_*(X,\bd X;G)$,
\item $I^{\hat p\to \hat q}H_*(\hat Z;G)\cong I^{\bar p\onto \bar q}H_*(Z,\bd Z;G)$,
\item $I^{\hat p\onto \hat q}H_*(\hat Y_i, \hat Z;G)\cong I^{\bar p\onto \bar q}H_*(Y_i,\bd Y_i;G).$
\end{enumerate}
Furthermore, these last isomorphisms preserve the intersection pairing when $G$ is a ring.
\end{lemma}
\begin{proof}
We will show the proof for $\hat Y_i$; the others are the same (though easier without the extra stratified boundary component). 

By \cite[Lemma 2.4]{GBF23}, we may assume $\hat p$ to be arbitrarily negative on $\hat W_i$. Therefore, it follows from the definition that no $\hat p$-allowable simplex can intersect $\hat W_i$. Thus $I^{\hat p}H_*(\hat Y_i;G)\cong I^{\hat p}H_*(\hat Y_i-\hat W_i;G)\cong I^{\bar p}H_*(Y_i-W_i;G)\cong I^{\bar p}H_*(Y_i;G)$, the last isomorphism by stratum-preserving homotopy equivalence. 

Next, by \cite[Lemma 2.4]{GBF23}, we might assume $\hat q$ to be arbitrarily large on $\hat W_i$. Thus there is no impediment to chains intersecting $W_i$. Thus in the neighborhood $N$ of $\hat W_i$ that is the product of $W_i$ with $(0,1]\supset \{1\}$, all allowable chains are homologous by product homologies to chains in $\hat W_i$. But as $\hat W_i$ consists entirely of singular strata, the coefficient system is $0$ there, and so $I^{\hat q}H_*(N;G)=0$ and similarly  $I^{\hat q}H_*(N,N\cap Z_0;G)=0$.
The isomorphism  $I^{\hat q}H_*(\hat Y_i,\hat Z;G)\cong I^{\bar q}H_*(Y_i,\bd Y_i; G)$ now follows by some easy arguments from the Mayer-Vietoris sequence for the pair consisting of $(N,N\cap Z_0)$ and $(\hat Y_i-\hat W_i, \hat Z-\hat W_i\cap \hat Z)$.  

The rest of the lemma follows easily.
\end{proof}

Let $\hat V=I^{\hat q/\hat p}H_{2n}(\hat Z;\Q)$ equipped with the anti-symmetric pairing $\Phi$. Let 
$$\hat A=\ker(I^{\hat q/\hat p}H_{2n}(\hat Z;\Q)\to I^{\hat q/\hat p}H_{2n}(\hat Y_1;\Q)),$$ 
$$\hat C=\ker(I^{\hat q/\hat p}H_{2n}(\hat Z;\Q)\to I^{\hat q/\hat p}H_{2n}(\hat Y_2;\Q)),$$
and 
$$\hat B=\ker(d: I^{\hat q/\hat p}H_{2n}(\hat Z;\Q)\to I^{\hat p}H_{2n-1}(\hat Z;\Q)).$$

\begin{corollary}\label{C: cor}
$$\sigma_{\bar p\onto \bar q}(X)=\sigma_{\bar p\onto \bar q}(Y_1)+\sigma_{\bar p\onto \bar q}(Y_2)+ \sigma(\hat V,\hat A,\hat B,\hat C).$$ 
\end{corollary}
\begin{proof}
The corollary follows from Theorem \ref{T: Wall} and the preceding lemma.
\end{proof}

It is reasonable to ask the following
question: Is it possible to identify $\sigma(\hat V;\hat A,\hat B,\hat C)$ as an invariant of a pairing involving only subspaces of intersection homology groups associated with $Z$? 
Unfortunately, the obvious choices do not seem to be correct. For example, $\sigma(\hat V,\hat A,\hat B,\hat C)$ cannot be the signature of the pairing $\Phi$ on $\im(I^{\bar q/\bar p}H_{2n}(Z)\to I^{\bar q/\bar p}H_{2n}(Z,\bd Z))$, which would be a natural guess. To see this, suppose all spaces are manifolds with boundary as in Wall's original non-additivity theorem. In this case, $Z$ is a manifold and $I^{\bar q/\bar p}H_{2n}(Z)=0$. Thus this term would always have to be $0$, which is certainly not true. Another natural guess would be that 
$\sigma(\hat V;\hat A,\hat B,\hat C)$ would be expressible in terms of a pairing on the intersection homology of $\bd Z$. However, this cannot be, as Theorem \ref{T: Wall} should be a special case of Corollary \ref{C: cor} in which all stratified boundaries (except for the intersection $Z$ itself) are empty. In such a case, any groups associated with $\bd Z$ would vanish, and this would violate the existence of the Maslov index term in Theorem \ref{T: Wall}. 

\begin{remark}
Rather than restratifying as we have done, it is tempting to do ``the usual thing'' and treat $\bd$-stratified pseudomanifolds  by simply adding a cone on the boundary and working with the resulting space. However, that will not quite do here, as $Z\cup_{\bd Z}c(\bd Z)$ will not generally be bicollared in $X\cup_{\bd X} c(\bd X)$. 

One alternative would be the following construction. Beginning with the bi-collared $Z\times [0,1]\subset X$, consider $X'=X\cup_{\bd Z\times [0,1]}(c(\bd Z)\times [0,1])$. Then $X'$ has stratified boundaries homeomorphic to $W_1 \cup_{\bd Z}c(\bd Z)$ and  $W_2\cup_{\bd Z}c(\bd Z)$. By separately coning off these stratified boundary components we get a space $X''$ that possesses strata $[0,1]\supset \{0,1\}$ and such that $X''-[0,1]$ is homeomorphic to the interior of $X$. If the perversities $\bar p$ and $\bar q$ are extended so that $\bar p$ takes values $<0$ on the new strata and $\bar q$ takes values $>\bar t$ on the new strata, then the intersection homology of $X''$ with these perversities is homeomorphic to that of $\hat X$ with respect to $\hat p$ and $\hat q$. 
\end{remark}

\section{Relation to Wall's non-additivity theorem}\label{S: manifold}

If we take our $\bd$-stratified pseudomanifolds to be $\bd$-manifolds then, as expected, we recover Wall's non-additivity theorem. The relationship between our Maslov index and Wall's is not completely obvious, as 
the pairing $\Phi$ requires interpretation from the manifold point of view. We will show that, in fact,  if $M$ is a $\bd$-manifold with non-empty manifold boundary and $X$ is the pseudomanifold obtained by coning off $\bd M$, then $I^{\bar q/\bar p}H_*(X)$ is just $H_{*-1}(\bd M)$ and the pairing $\Phi$ is the intersection pairing on $\bd M$, up to sign.

Suppose $M^m$ is a compact $\bd$-manifold with non-empty manifold boundary, $\bd M$. Let $X$ denote $M\cup_{\bd M} c(\bd M)$. We suppose $X$ is stratified as $X\supset v$, where $v$ is the cone vertex. Let $\bar p,\bar q$ be perversities for which $\bar p(v)<0$ and $\bar q(v)>m-2$. Then $I^{\bar p}H_*(X;G)\cong H_*(M;G)$ and $I^{\bar q}H_*(X;G)\cong H_*(M,\bd M;G)$. We would then expect from the long exact sequences of the pairs that $I^{\bar q/\bar p}H_*(X;G)\cong H_{*-1}(\bd M;G)$. We will make this isomorphism explicit.

Suppose $\xi$ is an $j$-chain in $\bd M$. Let $c\xi$ denote the chain obtained by coning off $\xi$ in $c(\bd M) \subset X$. In other words, if $\xi=\sum a_i \sigma_i$, then $c\xi=\sum a_i c(\sigma_i)$, where for a simplex $\sigma$, $c(\sigma):\Delta^{j+1}\to X$ is the cone on the map $\sigma:\Delta^j\to \bd M$ obtained by extending $\sigma$ linearly to the cone vertex. 
We take the convention that the new vertex is the \emph{first} vertex in $c\sigma$. With this convention, $\bd (c\xi)=\xi+c(\bd \xi)$. This coning $c$ determines a homomorphism $c: C_{*-1}(\bd M;G)\to I^{\bar q/\bar p}C_*(X;G)$, since every $c\xi$ is $\bar q$-allowable, as can be confirmed from the definition of allowability as $(c(\sigma))^{-1}(v)\subset \text{the $0$-skeleton of $\Delta^{j+1}$}$ for every singular simplex $\sigma$ in $\bd M$. Furthermore, $c$ is a chain map, as $\bd (c\xi)=\xi+c(\bd \xi)=c\bd \xi\in I^{\bar q/\bar p}C_*(X;G)$, since $\xi$ is $\bar p$-allowable. As a chain map, $c$ induces a homomorphism on homology, which we claim is an isomorphism.

\begin{lemma}\label{L: manifold}
The homomorphism $c$ induces an isomorphism $H_{*-1}(\bd M;G)\to I^{\bar q/\bar p}H_*(X;G)$.
\end{lemma}
\begin{proof}
Consider the diagram (with coefficients suppressed)

\begin{diagram}
&\rTo& H_{i+1}(M,\bd M)&\rTo& H_i(\bd M)&\rTo &H_i(M)&\rTo & H_i(M,\bd M) &\rTo\\
&&\dTo&&\dTo&&\dTo&&\dTo\\
&\rTo& I^{\bar q}H_{i+1}(X)&\rTo& I^{\bar q/\bar p}H_{i+1}(X)&\rTo &I^{\bar p}H_i(X)&\rTo & I^{\bar q}H_i(X) &\rTo\\
\end{diagram}
It will suffice to show this diagram commutes (up to sign). The map from $H_i(M)$ to $I^{\bar p}H_i(X)$ is given by inclusion, the map $H_i(M,\bd M)$ to $I^{\bar q}H_i(X)$ is given by taking a representative $\xi$ to $\xi-c(\bd \xi)$. 

It is easy to check the the squares on the right and in the middle commute. For the square on the left, note that if $\xi$ is a chain representing an element of $H_{i+1}(M,\bd M)$, then the image of $\xi-c(\bd \xi)$ in $I^{\bar q/\bar p}H_{i+1}(X)$ is simply $-c(\bd \xi)$ as $\xi$ is $\bar p$-allowable. This is enough to establish that the left square commutes up to sign. Thus the diagram commutes up to sign and has exact rows, which is enough to establish the isomorphism via the five-lemma.
\end{proof}

\begin{proposition}\label{P: manifold}
If $R$ is a ring and $M$ is a $4n-1$ $\bd$-manifold, the isomorphism $c$ of the preceding lemma takes   the intersection pairing $H_{2n-1}(\bd M;R)\otimes H_{2n-1}(\bd M;R)\to R$ to the pairing $-\Phi: I^{\bar q/\bar p}H_{2n}(X;R)\otimes I^{\bar q/\bar p}H_{2n}(X;R)\to R$, i.e. $[x]\pfb_{\bd M} [y]=-\Phi(c[x],c[y])$. 
\end{proposition}
\begin{proof}
Suppose $x\in C_i(\bd M;R)$, $y\in C_j(\bd M;R)$, represented by cycles in general position. Let $cx$ be the cone on $x$ as described above, and let $cy$ be the cone on $y$ except assuming that $y$ has first been pushed outward slightly into $c\bd M\subset X$ along the cone line so that $x$ does not intersect $cy$. In fact, we observe geometrically that $x\pfa_X cy=(\bd (cx))\pfa_X cy=0$, while $(cx)\pfa_X y=(cx)\pfa_X \bd (cy)$ must equal $x\pfa_{\bd M} y$ up to sign. This will establish the claim once we work out the sign.

We write out the argument in simplicial notation, which of course is not quite the actual situation, but it provides the correct intuition and reasoning.
With this abuse of notation, simplices of $cx$ have the form $[v,\sigma]=[v,v_0,\ldots, v_i]$, where $v$ is the singular point of $X$ and the $v_i$ are the vertices of $\sigma$, a simplex of $x$. The orientation here corresponds to a basis of vectors $[v,v_0],\ldots, [v,v_i]$. To compare with the orientation of $\sigma$, though, it is best to note that $[v,v_0,\ldots, v_i]=-[v_0,v,v_1,\ldots, v_i]$. Here $[v_0,v,v_1,\ldots, v_i]$ has an orientation corresponding to a basis of vectors $[v_0,v],[v_0,v_1],\ldots, [v,v_i]$, which is a basis for $\sigma$ with a vector from $v_0$ to $v$, which corresponds to an outward pointing normal from $M$, adjoined at the beginning. Thus, using our conventions from the Appendix, $\Phi(c[x],c[y])=\epsilon[cx\pfa_X \bd (cy)]=\epsilon[cx \pfa_X y]=\epsilon[-x\pfa_{\bd M}y]$.
\end{proof}

There is an alternative way to formulate the above correspondences using codimension one strata. In particular, instead of forming $X$, we can stratify $M$ as $M\supset \bd M$, where $\bd M$ is treated as a codimension one stratum of the stratified space $M$. If we then choose perversities $\bar p, \bar q$ such that $\bar p(\bd M)<0$, $\bar q(\bd M)\geq 0$, then we will again have $I^{\bar p}H_*(M;G)=H_*(M;G)$ and $I^{\bar q}H_*(M;G)=H_*(M,\bd M;G)$. This follows from \cite[Lemma 2.4]{GBF23}, which says that this is equivalent to choosing $\bar p(\bd M)$ arbitrarily negative and $\bar q(\bd M)$ arbitrarily large, and then simple arguments taking into account with the  stratified coefficient system. 

Using this alternative correspondence, we can recover Wall's non-additivity theorem \cite{Wa69}. In Wall's situation, we suppose 
$M^{4n}$ is a compact oriented $\bd$-manifold with manifold  boundary $W$ such that $M=M_1\cup M_2$, where $M_1,M_2$ are compact oriented $\bd$-manifolds and $M_1\cap M_2=\bd M_1\cap\bd M_2$ is a  $\bd$-manifold $N$   such that $\bd M_1=N-W_1$, $\bd M_2=B_2-N$, and $\bd N_1=\bd N_2=\bd N=P$.

Let $V=H_{2n-1}(P;\Q)$, and let $A,B,C$ be the respective kernels of the maps induced by inclusion from $V$ to $H_{2n-1}(W_1;\Q)$, $H_{2n-1}(N;\Q)$, $H_{2n-1}(W_2;\Q)$. For a $4n$ $\bd$-manifold, $\sigma(M)$ denotes the signature of the pairing on $\im(H_{2n}(M;\Q)\to H_{2n}(M,\bd M;\Q))$.

\begin{corollary}[Wall]
\label{cor:wall}
$$\sigma(M)=\sigma(M_1)+\sigma(M_2)-\sigma(V;A,B,C).$$
\end{corollary}
\begin{proof}
Wall's theorem follows from our Theorem \ref{T: Wall} as follows. Restratify $M$ as $M\supset \bd M$ and choose $\bar p(\bd M)$ arbitrarily negative and $\bar q(\bd M)=\bar t(\bd M)-\bar p(\bd M)=-1-\bar p(\bd M)$. Then, $I^{\bar p}H_*(M;G)=H_*(M;G)$,  $I^{\bar q}H_*(M;G)=H_*(M,\bd M;G)$, $I^{\bar p}H_*(M_i;G)=H_*(M_i;G)$, and  $I^{\bar q}H_*(M_i,N;G)=H_*(M_i,\bd M_i;G)$ . In particular, $I^{\bar p\to \bar q}H_*(M;\Q)\cong \im(H_*(M;\Q)\to H_*(M,\bd M;\Q))$, and $I^{\bar p\onto \bar q}H_*(M_i,N;\Q)\cong \im(H_*(M_i;\Q)\to H_*(M_i,\bd M_i;\Q))$. Furthermore, by Lemma \ref{L: manifold} and Proposition \ref{P: manifold}, $I^{\bar q/\bar p}H_*(N;\Q)\cong H_{*-1}(P;R)$ with $\Phi$ corresponding to the negative of the intersection pairing on $P$. The corollary thus follows from Theorem \ref{T: Wall}. 
\end{proof}

\section{Calculational tools and examples}

In this section, we provide some calculational tools for perverse signatures by applying our (non-)additivity theorem and use these to calculate some examples of perverse signatures.  Some of 
our tools are versions of standard results in the signature package for manifolds.

The first tool is a version of Novikov additivity for perverse signatures.

\begin{corollary}\label{C: novikov}
With the hypotheses of Theorem \ref{T: Wall}, suppose $I^{\bar p}H_{2n}(Z;\Q)\to I^{\bar q}H_{2n}(Z;\Q)$ is surjective and $I^{\bar p}H_{2n-1}(Z;\Q)\to I^{\bar q}H_{2n-1}(Z;\Q)$ is injective. Then $$\sigma_{\bar p\to\bar q}(X)=\sigma_{\bar p\onto\bar q}(Y_1)+\sigma_{\bar p\onto\bar q}(Y_2),$$ as in Novikov's additivity theorem. In particular, Novikov additivity holds if $Z$ is a closed oriented manifold with trivial stratification.
\end{corollary}
\begin{proof}
In this case, $I^{\bar q/\bar p}H_{2n}(Z;\Q)=0$ by the long exact sequence relating $\bar p$ and $\bar q$ intersection homology. Thus $V$ and hence $\sigma(V;A,B,C)$ are trivial. 
\end{proof}
From this, we can recover Siegel's theorem regarding Novikov additivity of Witt spaces \cite{Si83}. Indeed, when $X,Y_i,Z$ are all Witt-spaces, $I^{\bar m}H_*\cong I^{\bar n}H_*$ for each, and thus $\sigma_{\bar m\to\bar n}(X)=\sigma_{\bar m\onto\bar n}(Y_1)+\sigma_{\bar m\onto\bar n}(Y_2)$. These signatures $\sigma_{\bar m\to\bar n}(X)$ and $\sigma_{\bar m\to\bar n}(Y_1)$ are just the signatures of the middle-perversity middle-dimension intersection pairings on these Witt spaces \cite{Si83}. 

A weak form of the cobordism invariance also follows from Corollary \ref{C: novikov}.
Let $\Sigma_X$ denote the union of singular strata of the stratified pseudomanifold $X$, and let $N(\Sigma_X)$ denote the (closed) regular neighborhood of $\Sigma_X$ in $X$. 
Let $W$ be a compact $\bd$-stratified pseudomanifold whose stratified boundary is the disjoint union $X\amalg -Y$. Suppose further that $N(\Sigma_X)\cong N(\Sigma_Y)$ and that $N(\Sigma_W)\cong N(\Sigma_X)\times I$ with $N(\Sigma_X)\times 1$ identified with $N(\Sigma_X)$ and  $N(\Sigma_X)\times 0$ identified with $N(\Sigma_Y)$. We will refer to such a $W$ as a bordism rel $\Sigma$ from $X$ to $Y$ and say that $X$ and $Y$ are cobordant rel $\Sigma$.

\begin{proposition}
If $X$ and $Y$ are s-closed $4n$-dimensional stratified pseudomanifolds  that are cobordant rel $\Sigma$, then $\sigma_{\bar p\to \bar q}(X)=\sigma_{\bar p\to \bar q}(Y)$. 
\end{proposition} 
\begin{proof}
The pseudomanifolds $X$ and $Y$ can be decomposed, respectively, as $M\cup_{\bd M} N(\Sigma_X)$ and $M'\cup_{\bd M'} N(\Sigma_Y)$, where $M$ and $M'$ are manifolds. Thus by Corollary \ref{C: novikov}, $\sigma_{\bar p\to \bar q}(X)=\sigma(M)+\sigma_{\bar p\onto \bar q}(N(\Sigma_X))$ and $\sigma_{\bar p\to \bar q}(Y)=\sigma(M')+\sigma_{\bar p\onto \bar q}(N(\Sigma_Y))=\sigma(M')+\sigma_{\bar p\onto \bar q}(N(\Sigma_X))$. Thus is suffices to show $\sigma(M)=\sigma(M')$. But $\bd (W-\text{int}(N(\Sigma_W)))=M\cup_{\bd M} (\bd M\times I)\cup_{-\bd M'} -M'\cong M\cup_{\bd M} -M'$. Thus $0=\sigma(M\cup_{\bd M} -M')=\sigma(M)-\sigma(M')$, by ordinary Novikov additivity and the bordism invariance of manifold signatures. 
\end{proof}

The resulting cobordism group is infinite dimensional since each possible boundary neighborhood yields at least
one distinct cobordism class. Because it does not permit cobordisms that change
a neighborhood of the singular stratum, it is not really in the same vein as the 
cobordism invariants known for manifolds, Witt spaces,
and Banagl non-Witt spaces, which play important roles in the signature packages in
those settings.
We are hopeful, however, that it may be possible in the future to define a set of spaces
for which various perverse signatures satisfy a better cobordism invariance.
 
 The next tool is a version of the standard multiplicativity of signatures.

\begin{lemma}\label{mult}
Suppose $Y$ is an s-closed oriented $4k$-dimensional pseudomanifold and that $N$ is a closed oriented $4n$-dimensional manifold. Then for perversities $\bar p\leq \bar q$, $\bar p+\bar q=\bar t$, we have $\sigma_{\bar p\to \bar q}(N\times Y)=\sigma(N)\sigma_{\bar p\to \bar q}(Y)$. 
\end{lemma}
\begin{proof}
By the K\"unneth theorem for intersection homology in which one term is a manifold (see \cite{Ki}, which extends to more general perversities and stratified coefficients), $I^{\bar p}H_*(N\times Y;\Q)\cong  H_*(M;\Q)\otimes I^{\bar p}H_*(Y;\Q)$, and similarly for $\bar q$. Thus, by the naturality of the K\"unneth theorem, $I^{\bar p\to \bar q}H_*(N\times Y;\Q)\cong H_*(N;\Q)\otimes I^{\bar p\to \bar q}H_*(Y;\Q)$. The lemma now follows just as it does for manifolds (e.g. \cite{Hirz, CHS57}), using stratified general position arguments to see that the intersection pairing of the product behaves as one expects. 
\end{proof}

This allows us to construct a nontrivial example of a perverse signature that is neither a Witt signature nor an example of one of Banagl's non-Witt signatures.

\begin{example}(A nontrivial perverse signature of a space that is neither Witt nor non-Witt).  
Suppose $W$ is a compact oriented $4k$-dimensional $\Q$-Witt space with non-zero Witt signature. Let $M$ be a $4m$-dimensional connected compact oriented PL $\bd$-manifold with non-empty  boundary $\bd M$. Consider the space $X= M\times W\cup_{\bd M\times W} \bd M\times \bar cW$, i.e. the space obtained from $M\times W$ by coning off the stratified boundary fiberwise.  Then $X$ is not a Witt space, as $W$ is the link of the stratum  $\bd M\times v$, where $v$ is the cone point of the closed cone $\bar c W$, and by assumption, $W$ has non-vanishing middle-dimensional middle-perversity intersection homology. Furthermore, because the signature of $W$ does not vanish, $X$ cannot be a Banagl ``non-Witt'' space. Nonetheless, the signature $\sigma_{\bar m\to\bar n}(X)$ is defined, and  we will show that if $\bd M$ is PL homeomorphic to $S^{4m-1}$,  then $\sigma_{\bar m\to\bar n}(X)=\sigma(M)\sigma(W)$, where $\sigma(M)$ is the usual manifold signature of $M$ and $\sigma(W)=\sigma_{\bar m\to \bar n}(W)$ is the Witt signature of $W$.  By choosing appropriate $M$ and $W$, we can of course arrange for this
to be nontrivial.

Let $\hat M$ be the closed manifold $M\cup_{S^{4m-1}}D^{4m}$. By Lemma \ref{mult},  $\sigma_{\bar m\to\bar n}(\hat M\times W)=\sigma(M)\sigma_{\bar m\to \bar n}(W)=\sigma(M)\sigma(W)$, the last equality because $W$ is Witt. 
Notice that $S^{4m-1}\times  W$ is also a Witt space and so $I^{\bar m}H_*(W;\Q)\cong I^{\bar n}H_*(W;\Q)$. Thus, by Corollary \ref{C: novikov}, $\sigma_{\bar m\to\bar n}(\hat M\times W)=\sigma_{\bar m\onto\bar n}(M\times W)+\sigma_{\bar m\onto\bar n}(D^{4m}\times W)$. But $D^{4m}\times W$ possesses an orientation-reversing self-homeomorphism, so $\sigma_{\bar m\onto\bar n}(D^{4m}\times W)=0$, and 
$\sigma_{\bar m\onto\bar n}(M\times W)$ is the Witt signature $\sigma_{\bar m\onto\bar n}(M\times W)=\sigma(M\times W)$. Thus $\sigma(M\times W)=\sigma(M)\sigma(W)$. 

Returning now to our space $X$, obtained by coning off the stratified boundary of $M\times W$ fiberwise, we see by a second application of Corollary \ref{C: novikov} that $\sigma_{\bar m\to\bar n}(X)=\sigma(M\times W) + \sigma_{\bar m\to \bar n}(S^{4n-1}\times \bar cW)$. But $S^{4n-1}\times \bar cW$ again possesses an orientation-reversing self-homeomorphism, so its perverse signature is $0$. Putting the preceding arguments together, we obtain 
$$
\sigma_{\bar m\to\bar n}(X)=\sigma(M\times W)=\sigma(M)\sigma(W).$$
\hfill\qedsymbol
\end{example}

The next examples are similar to results for standard signatures.

\begin{example}(Pseudomanifolds with involutions).
If $(Y_1,Z)$ is homeomorphic to $(-Y_2,-Z)$ rel $Z$ (i.e. by an isomorphism that fixes $Z$ pointwise), then $\sigma_{\bar p\to\bar q}(X)=0$.
We can see this as follows.  Since $Y_1\cong -Y_2$,  their signatures are the negatives of each other. In addition, with the hypotheses, it is clear that the inclusions $I^{\bar q/\bar p}H_{2n}(Z;\Q)\to I^{\bar q/\bar p}H_{2n}(Y_1;\Q)$ and $I^{\bar q/\bar p}H_{2n}(-Z;\Q)\to I^{\bar q/\bar p}H_{2n}(Y_2;\Q)$ are isomorphic maps with identical kernels (this is why we require $Z$ to be fixed by the homeomorphism). Therefore $A=C$. But an odd permutation of the subspaces of $A,B,C$ alters $\sigma(V;A,B,C)$ by a sign. Hence $\sigma(V;A,B,C)=0$. 
\hfill\qedsymbol
\end{example}
\begin{example}\label{C: suspension}(Suspensions).
If $X$ is the suspension of the s-closed stratified pseudomanifold $Z$, then $\sigma_{\bar p\to\bar q}(X)=0$.  This follows from the preceeding example by taking $Y_1$ and $Y_2$ to be the two cones on $Z$.

This example can also be obtained with less machinery by observing that if $\bar p\leq \bar q$ and $\bar p+\bar q=\bar t$, then in fact $\bar p\leq \bar m\leq \bar n\leq \bar q$, where $\bar m,\bar n$ are the lower- and upper-middle perversities. Thus $I^{\bar p}H_*(X;\Q)\to I^{\bar p}H_*(X;\Q)$ factors through $I^{\bar m}H_*(X;\Q)$ and $I^{\bar n}H_*(X;\Q)$. But for a $4k$-dimensional suspension $X$, $I^{\bar m}H_{2k}(X;\Q)=I^{\bar n}H_{2k}(X;\Q)=0$. \hfill\qedsymbol
\end{example}

The next lemma is a vanishing theorem for perverse signatures of cones.

\begin{lemma}[Perverse signatures of cones]
Let $\bar p\leq \bar q$, $\bar p+\bar q=\bar t$, and suppose $Y$ is an s-closed, oriented $4n-1$ dimensional pseudomanifold with closed cone $\bar cY$. Then $\sigma_{\bar p\onto \bar q}(\bar cY)=0$.
\end{lemma}
\begin{proof}
Notice that $\bar cY$ is a $\bd$-stratified pseudomanifold with  $\bd (\bar cY)=Y$. By the standard cone formula for intersection homology, for any perversity $\bar r$, 
$I^{\bar r}H_*(\bar cY;\Q)$ is either $0$ or isomorphic to $I^{\bar r}H_{2n}(Y;\Q)$, with the isomorphism determined by inclusion $Y\into \bar cY$. Thus  $I^{\bar p\onto \bar q}H_{2n}(\bar cY, Y;\Q)=0$, because any possible non-zero element $[x]\in I^{\bar p}H_{2n}(\bar cY;\Q)$ can be written with the support of $x$ in $Y$, and so the image of $[x]$ is $0$ in $I^{\bar q}H_{2n}(\bar cY, Y;\Q)$.  Thus certainly $\sigma_{\bar p\onto \bar q}(\bar cY)=0$.
\end{proof}

We use this lemma in the next example together with an explicit computation of the Maslov index term of Theorem \ref{T: Wall} to show that coning off a stratified boundary does not change perverse signatures, i.e.
that $\sigma_{\bar p\to \bar q}(X\cup_{\bd X}\bar c(\bd X))=\sigma_{\bar p\onto \bar q}(X)$. Here, $X$ may be neither Witt nor Banagl non-Witt. 

\begin{example}(Coning boundaries).
Suppose $\hat X$ is an s-closed oriented $4n$-pseudomanifold of the form $\hat X= X\cup_{\bd X}\bar c(\bd X)$. Let $\bar p\leq \bar q$, $\bar p+\bar q=\bar t$ be two perversities on $\hat X$. Then $\sigma_{\bar p\to \bar q}(\hat X)=\sigma_{\bar p\onto \bar q}(X)$. 

To show this, we apply Theorem \ref{T: Wall}. By the preceding lemma, $\sigma_{\bar p\onto \bar q}(\bar c(\bd X))=0$, so we need only show that the Maslov index term vanishes. 
Consider $\ker(d: I^{\bar q/\bar p}H_{2n}(\bd X;\Q)\to  I^{\bar p}H_{2n-1}(\bd X;\Q))$, which is the group $B$ in the index term. By the exact sequence \eqref{E: LES}, $B=\im (I^{\bar q}H_{2n}(\bd X;\Q)\to I^{\bar q/\bar p}H_{2n}(\bd X;\Q))$, and so every class $[x]\in B$ can be represented by a $\bar q$-allowable $2n$-cycle in $\bd X$. Consider now $(i_{\bd X\subset \bar c(\bd X)})_*[x]\in I^{\bar q/\bar p}H_{2n}(\bar c(\bd X);\Q)$. This is also represented by the same chain $x$. 
In $\bar c(\bd X)$, we have that $\bd (\bar cx)=\pm x$, so if we can show $\bar cx$ is 
$\bar q$ allowable, then we have $(i_{\bd X\subset \bar c(\bd X)})_*[x]=0$, which implies $B\subset A$, where $A=\ker((i_{\bd X\subset \bar c(\bd X)})_*:I^{\bar q/\bar p}H_{2n}(\bd X;\Q)\to I^{\bar q/\bar p}H_{2n}(\bar c(\bd X);\Q))$. Thus $B\cap (A+C)=B\cap A$, and so $W=\frac{B\cap (A+C)}{B\cap A+B\cap C}=0$, which implies that the Maslov index term must vanish. 

To show $\bar cx$ is $\bar q$-allowable, first note that the conditions on  $\bar p,\bar q$ imply that $\bar q\geq \bar n$, where $\bar n$ is the upper middle
perversity. For any simplex of $\bar cx$, we only need to check allowability at the cone vertex $v$ (the allowability of $\bar c x$ otherwise comes for free - see the arguments in \cite{GBF10}). For simplices $\sigma$ of $\bar cx$ that intersect the cone vertex, we know that $\sigma^{-1}(v)$ is in the $0$-skeleton of the model $\Delta^{2n+1}$. So by definition of allowability, we only need to check that $0\leq 2n+1-4n+\bar q(v)=1-2n+\bar q(v)$. But $\bar q(v)\geq \bar n(v)=2n-1$. So $1-2n+\bar q(v)\geq 0$, and $\bar q$-allowability is confirmed. 
\hfill\qedsymbol
\end{example}

It is somewhat unsatisfying that the Maslov index in the previous example is trivial, so we would also like to show that it is not always.  The following example does this.

\begin{example}(A nontrivial Maslov index).
Let $\mf D$ be the unit tangent disk bundle over $S^{2n}$. Let $N=[-1,1]\times S^{2n-1}$ be a neighborhood of the equator of $S^{2n}$. The restriction of $\mf D$ over $N$ is a trivial disk bundle $ [-1,1]\times S^{2n-1}\times B^{2n}$, where $B^{2n}$ is the $2n$-disk. Now, for each $t\in[-1,1]$, to $t\times S^{n-1}\times B^{2n}$ we adjoin the cone on $t\times S^{n-1}\times  \bd B^{2n}$. In other words, we form $W=\mf D\cup_{[-1,1]\times S^{2n-1}\times \bd  B^{2n}} ([-1,1] \times \bar c(S^{2n-1}\times  \bd  B^{2n}))$. Another way to say this is that $W$ is the union of two spaces, one of which is the product of $[-1,1]$ with the Thom space of the trivial $\R^{2n}$-bundle over $S^{2n-1}$ and the other of which consists of  the tangent disk bundles over the caps $S^{2n}-([-1,1]\times S^{2n-1})$. Next, we note that $W$ has a boundary consisting of two pieces. One boundary piece  is the union of the boundary of the tangent disk bundle over the top cap of $S^{2n-1}$ with the cone on $1\times S^{2n-1}\times \bd B^{2n}$, and the other consists of the union of the tangent disk bundle over the bottom cap with the cone over $-1\times S^{2n-1}\times \bd B^{2n}$. Let $X$ be the union of $W$ with two cones, one on each boundary piece. Then $X$ is a normal compact pseudomanifold of dimension $4n$. It can be stratified by $X^{2n}\supset X^1\supset X^0$. The $0$-stratum $X^0$ consists of the cone vertices of the last two cones adjoined in the formation of $X$. The $1$-stratum $X^1$ consists of the union of $[-1,1]\times v$, where $v$ is the cone vertex of the cone on $S^{2n-1}\times \bd B^{2n}$, with its extension into the capping cones. 
$X$ can be oriented with an orientation consistent with one chosen on $\mf D$. In fact, for $n>1$, $X$ and $X-X^1$ will be simply connected. 

Let $\bar p$ be the $0$ perversity, and let $\bar q$ be the top perversity. To compute $I^{\bar p}H_{2n}(X)$, we recall that  a PL $i$-chain will be $\bar p$ allowable with respect to the stratum $X^{4n-k}$ only if its intersection with that stratum has dimension $\leq i-k+\bar p(k)$ (and similarly for the boundary). In this case, the relevant $i$ will be $2n$ or $2n+1$ and $k$ will be $4n$ or $4n-1$. With $\bar p$ being the $0$ perversity, the implication is that, if $n$ is sufficiently large, no chains of dimension near $2n$ will be able to intersection the singular strata. Thus $I^{\bar p}H_{2n}(X)\cong H_{2n}(X-X^1)$. But $X-X^1$ is easily seen from the construction to retract back to $\mf D$, which retracts to $S^{2n}$ itself. So $I^{\bar p}H_{2n}(X)\cong \Q$. On the other hand, to compute $I^{\bar q}H_{2n}(X)$, we recall that $\bar q(4n-1)=\bar t(4n-1)=4n-3$ and $\bar q(4n)=4n-2$. For large $n$, we see that all chains in degrees near $2n$ will be completely allowable (since the dimensions of their intersection with $X^1$ and $X^0$ cannot exceed $1$), and so $I^{\bar q}H_{2n}(X)=H_{2n}(X)$. Since $X^1$ is contractible, this is isomorphic to $H_{2n}(X, X^1)$, which, furthermore, by homotopy equivalence and excision, is isomorphic to $H_{2n}(X,X-S^{2n})\cong H_{2n}(\mf D, \mf D-S^{2n})$. This is just the homology of the Thom space. So, the inclusion $I^{\bar p}H_{2n}(X)\to I^{\bar q}H_{2n}(X)$ corresponds to the inclusion of $H_{2n}(S^{2n})$ into the Thom space of its tangent bundle. Here it is well known that the intersection of the generator of  $H_{2n}(S^{2n})$ with its image in the homology of the Thom space will be represented by the euler number of $S^{2n}$ in $H_0(S^{2n})=\Q$. For an even dimensional sphere this number is $2$. Hence the perverse $\bar p,\bar q$ signature of $X$ is $1$. 

Now, let us decompose $X$ into two pieces along the codimension $1$ sub-pseudomanifold $Y=0\times (S^{2n-1}\times B^{2n})\cup_{0\times S^{2n-1}\times \bd B^{2n}} \bar c(S^{2n-1}\times \bd B^{2n})$. This decomposes $X$ into two identical pieces, say $Z$ and $Z'$, each constructed over one hemisphere of $S^{2n}$. We let $Z,Z',Y$ inherit their stratifications (and perversities) from $X$. Now, consider $I^{\bar p}H_{2n}(Z)$. By the same arguments as above, $I^{\bar p}H_{2n}(Z)\cong H_{2n}(Z-Z\cap X^1)$. But $Z-Z\cap X^1$ retracts to the piece of $\mf D$ over the hemisphere of $S^{2n}$, which retracts to that hemisphere, itself. So $I^{\bar p}H_{2n}(Z)=0$, and the perverse signature of each piece must vanish.

We thus see that the Maslov index term for the given decomposition of $X$ must be non-zero. (Alternatively, it would have been sufficient to note that the signature of $X$ is $1$, which is odd, but that by symmetry $Z$ and $Z'$ must have identical signatures  mod $2$.) \hfill\qedsymbol
\end{example}

Our final example relates the Maslov index terms in the non-additivity formula
to the $\tau$ and $\tau_i$ invariants defined for fiber bundles in \cite{Dai} and \cite{Hun07}.
These measure what Dai calls non-multiplicativity of the signature, and they relate in analysis  to the pairings on certain 
noncompact manifolds of  harmonic $L^2$ forms that are exact but that are not $d$ of any $L^2$ form.

\begin{example}(Bundles and $\tau$ invariants).
Let $Y$ be the total space of a compact fiber bundle $F \hookrightarrow  Y \rightarrow B$.
Assume that $Y$ is $4k-1$ dimensional.  Form $X$ by coning off the fibers of $Y$. 
Then $X$ is a stratified pseudomanifold with one singular stratum homeomorphic 
to $B$ and of codimension $f+1$, where ${\rm dim}(F) = f$.
Thus only the values of perversities at codimension $f+1$ are relevant.
Assume  $p=\bar p(f+1) = \bar m(f+1) - j$ for some non-negative integer $j$,  where $\bar m$ denotes
the lower middle perversity.  Let $\bar q$ be the dual perversity to $\bar p$.

In the language of \cite{Hun07}, which uses cohomological indexing and notation, our perverse signature defined  on $\im(I^{\bar p}H_{2k}(X;\R)\to   I^{\bar q}H_{2k}(X,\bd X;\R))$ instead appears as a signature of a pairing on  $\im(IH^{2k}_{\bar p,0}(X,Y)\to   IH^{2k}_{\bar q}(X))$. The groups
$IH^{2k}_{\bar p,0}(X,Y)$ and $IH^{2k}_{\bar q}(X)$ are computed\footnotemark as hypercohomology groups of complexes of appropriately defined $L^2$ forms on the regular part of $X$, and the pairing is defined by integrating the exterior product of forms over the regular part of $X$. 
\footnotetext{For the purposes of comparison, we note that if $\hat X\cong X\cup_Y\bar cY$ and $v$ denotes the cone vertex, then  $IH^{2k}_{\bar p,0}(X,Y)$  corresponds to the hypercohomology of the Deligne sheaf on $\hat X$ with perversity value $p$ on $B$ and $-1$ on $v$, while $IH^{2k}_{\bar q}(X)$ corresponds to the hypercohomology of the Deligne sheaf on $\hat X$ with the dual perversity values.}

Then from \cite{Hun07}, we know that the perverse signatures of $X$ are 
calculated by:
\[
\sigma_{\overline{p}\onto\overline{q}}(X)=  \sum_{i=2+2j}^{\infty} \tau_i,
\]
where $\tau_i$ is calculated from the $i$th pages of the $\bar p$ and $\bar q$ truncated 
Leray spectral sequences for the cohomology of the fiber bundle $Y$ as the signature of the form:
\[
\begin{array}{llll}
\sigma_i: &  E^p_i \otimes E^q_i & \longrightarrow & \mathbb{R},\\
                  & \phi \otimes \psi & \longrightarrow & (\phi \cdot d_i \psi, \beta_i ),
\end{array}                 
\]
\noindent where $\beta_i$ is the volume element on the $i$th page.

On the other hand, we can decompose $B$ into an arbitrary number of 
contractible polygons, $P_j$, and lift this decomposition to a decomposition of $X$ as an 
arbitrary number of pieces of the form $\overline{c}F \times P_j$.  
But the perverse signatures on such pieces are trivial:

\begin{lemma}
Let $B$ be a closed euclidean ball, $F$ a compact pseudomanifold, and $W=B\times \bar cF$. 
Then $\sigma_{\bar p\onto \bar q}(W)=0$.
\end{lemma}
\begin{proof}
By stratum-preserving homotopy equivalence, $I^{\bar p}H_*(W;\R)\cong 
I^{\bar p}H_*(cF;\R)$, which is either $0$ or $I^{\bar p}H_*(F;\R)$, 
with any non-zero elements represented by chains on $F$. Such chains 
clearly represent trivial elements in $I^{\bar q}H_*(W,\bd W;\R)\cong 
I^{\bar q}H_*(W,(B\times F)\cup(\bd B\times cF);\R)$. Thus 
$\im(I^{\bar p}H_*(F;\R)\to I^{\bar q}H_*(W,\bd W;\R)=0$ in all degrees.
\end{proof}
Thus either the Maslov indices that arise in decomposing $X$ in this fashion are
nontrivial, or we get a remarkable vanishing of the $\tau_i$ for fiber bundles.  In the
case that the fiber is a sphere, this comes back to Wall non-additivity of the signature
for manifolds with boundary, which is of course generally nontrivial.  For example,
in the case of the Hopf fibration of $S_3$, one can directly calculate that $\tau_2= -1$;
see e.g. \cite{HHM}.
It would be surprising if the $\tau_i$ were only nontrivial for fiber bundles with spherical 
fibers, so we expect rather that the Maslov terms are generally nontrivial.\hfill\qedsymbol

\end{example}

\appendix

\section{Orientations and intersection numbers}

In this appendix we establish conventions for orientation and intersection numbers. This is not meant to be a thorough treatise on every possible case that can occur in the stratified world, but rather the working through of the simplest manifold cases in order to establish compatibility of convention choices.

Let $M$ be an $m$-dimensional oriented $\bd$-manifold. We choose the orientation of $\bd M$ by adjoining an outward-pointing normal in the first component, i.e. if $x\in \bd M$, $e_1,\ldots, e_{m-1}$ is a basis for $T_x\bd M$, and $n\in T_xM$ is an ``outward pointing'' vector, then the ordered collection $\langle e_1,\ldots, e_{m-1}\rangle$ agrees with the orientation for $\bd M$ if and only if $\langle n, e_1,\ldots, e_{m-1}\rangle$ agrees with the orientation for $M$. This convention seems to agree with the standard conventions for simplices.

Suppose $\xi,\eta$ are cycles of complementary dimension in $\bd M$ in general position and intersecting generically at the point $x$. Then the contribution to the intersection number $\epsilon[\xi\pfa \eta]$ of the intersection at $x$ is $\pm 1$ according to whether a local basis for $\xi$ concatenated with a local basis for $\eta$ agrees or disagrees with the orientation at $T_x\bd M$. It makes sense to talk about local bases for $\xi$ and $\eta$ as generic intersections will occur in the interiors of oriented simplices. 

Suppose now that there is a chain $\Xi$ in $M$ with $\bd \Xi=\xi$ contained in $\bd M$. We may assume that in a neighborhood of $\bd M$, $\Xi$ looks like the chain $[0,1]\times \xi$ with the ``1'' end of the cylinder on the boundary (suitably simplicialized). Note that this gives the proper boundary $\bd ([0,1]\times \xi)=1\times \xi-0\times \xi$ with $1\times \xi=\xi\subset \bd M$. Note also that the $[0,1]$ component points in the direction of an outward pointing normal. Thus if $\xi$ and $\eta$ intersect at $x$, the intersection number contribution at $x$ of $\epsilon[\xi\pf \eta]$ in $\bd M$ is equal to the intersection number contribution at $x$ of $\Xi$ and $\eta$ in $M$. This is because the intersection number of $\Xi$ with $\eta$ in $M$ is determined by using the basis of $[0,1]\times \xi$ (i.e. the outward normal and then the basis of $\xi$) and then the basis for $\eta$. Since the normal comes at the beginning, there is agreement with how we expect to compare orientations in $M$ with those in $\bd M$. On the other hand, suppose $H$ is a chain in $M$ with $\bd H=\eta$ and that $H$ looks like  $[0,1]\times \eta$ in a neighborhood of $\bd M$. Then the intersection at $x$ of $\xi$ with $H$ compares with the basis for $M$ the basis obtained from $\xi$ then from the outward normal, then from $\eta$. So to compare properly with the intersection number of $\xi$ and $\eta$ in $\bd M$, we must move the normal to the front. This changes the orientation number by $(-1)^{|\xi|}$. So the intersection number of $\xi$ with $\eta$ in $\bd M$ is $(-1)^{|\xi|}$ times the intersection number of $\xi$ with $H$ in $M$. 

Summarizing, we have:
\begin{align*}
\bd \Xi\pfa_{\bd M}\eta&=\Xi\pfa_M \eta\\
\xi\pfa_{\bd M}\bd H &= (-1)^{|\xi|} \xi\pfa_M \eta.
\end{align*}

\providecommand{\bysame}{\leavevmode\hbox to3em{\hrulefill}\thinspace}
\providecommand{\MR}{\relax\ifhmode\unskip\space\fi MR }
\providecommand{\MRhref}[2]{%
  \href{http://www.ams.org/mathscinet-getitem?mr=#1}{#2}
}
\providecommand{\href}[2]{#2}

Some diagrams in this paper were typeset using the \TeX\, commutative
diagrams package by Paul Taylor.


\begin{thebibliography}{10}

\bibitem{ALMP} P. Albin, Leichtmann, R. Mazzeo, P. Piazza, 
\emph{The signature package on Witt spaces, I. Index classes}, arXiv:0906.1568.

\bibitem{APS} M.F. Atiyah, V.K. Patodi, I.M. Singer, \emph{Spectral asymmetry and Riemannian
Geometry, I}, Math. Proc. Cambridge Philos. Soc. {\bf 77} (1975), 43-69.

\bibitem{AS} M.F. Atiyah, I.M. Singer, \emph{The index of elliptic operators. III},
Ann. of Math. (2) 87 1968 546--604. 

\bibitem{BaIH}
Markus Banagl. {\bf Topological invariants of stratified spaces}. \emph{Springer
  Monographs in Mathematics}, Springer-Verlag, New York, 2006.

\bibitem{BBW} B. Boo\ss-Bavnbek, K.P. Wojciechowski. {\bf Elliptic
Boundary Problems for Dirac Operators}. \emph{Mathematics:  Theory and Applications},
Birkh\"auser, Boston, 1993.

\bibitem{Bo}
A.~Borel~et. al. {\bf Intersection cohomology}. \emph{Progress in Mathematic},
  vol.~50, Birkhauser, Boston, 1984.

\bibitem{BHS}
J.P. Brasselet, G.~Hector, and M.~Saralegi, \emph{The\'eor\`eme de de{R}ham
  pour les vari\'et\'es stratifi\'ees}, Ann. Global Anal. Geom. \textbf{9}
  (1991), 211--243.

\bibitem{BW1}
William Browder. {\bf Surgery on Simply-Connected Manifolds}. Springer-Verlag, New York-Heidelberg-Berlin, 1972.


\bibitem{Bry}
Jean{-L}uc Brylinski, \emph{Equivariant intersection cohmology}, Contemp. Math.
  \textbf{139} (1992), 5--32.

\bibitem{CLM} 
Sylvain~E. Cappell, Ronnie Lee, and Edward~Y. Miller, \emph{On the {M}aslov index}, Comm. Pure Appl. Math. \textbf{47} (1994), 121--186.


\bibitem{CS}
Sylvain~E. Cappell and Julius~L. Shaneson, \emph{Singular spaces,
  characteristic classes, and intersection homology}, Annals of Mathematics
  \textbf{134} (1991), 325--374.

\bibitem{Car1}
G.~Carron, \emph{{$L\sp 2$}-cohomology of manifolds with flat ends}, Geom.
  Funct. Anal. \textbf{13} (2003), no.~2, 366--395.

\bibitem{Car3}
Gilles Carron, \emph{Cohomologie {$L^2$} des vari{\'e}t{\'e}s {QALE}},
  Preprint, 2005.

\bibitem{Car2}
\bysame, \emph{Cohomologie {$L\sp 2$} et parabolicit\'e}, J. Geom. Anal.
  \textbf{15} (2005), no.~3, 391--404.

\bibitem{Chee80}
J.~Cheeger, \emph{On the {H}odge theorey of {R}iemannian pseudomanifolds},
  Geometry of the Laplace Operator (Proc. Sympos. Pure Math., Univ. Hawaii,
  Honolulu, Hawaii, 1979) (Providence, RI), vol.~36, Amer. Math. Soc., 1980,
  pp.~91--146.
  
  \bibitem{CD} Jeff Cheeger and Xianzhe Dai, {\em $L\sp 2$-cohomology of spaces with nonisolated conical singularities and nonmultiplicativity of the signature}, Riemannian topology and geometric structures on manifolds, 1--24, Progr. Math., 271, BirkhŠuser Boston, Boston, MA, 2009. 

\bibitem{CGM82}
Jeff Cheeger, Mark Goresky, and Robert MacPherson, \emph{$L^2$-cohomology and
  intersection homology of singular algebraic varieties.}, Seminar on
  Differential Geometry (Princeton, NJ), Ann. of Math. Stud., vol. 102,
  Princeton Univ. Press, 1982, pp.~303--340.
  
\bibitem{CHS57}
S.~S. Chern, F.~Hirzebruch, and J.-P. Serre, \emph{On the index of a fibered
  manifold}, Proc. Amer. Math. Soc. \textbf{8} (1957), 587--596.  
  
\bibitem{Dai} X. Dai, \emph{Adiabatic limits, nonmultiplicativity of signature, and the Leray
spectral sequence}, J. Amer. Math. Soc. {\bf 4} (1991), 265-321.

\bibitem{Dold}
Albrecht Dold. {\bf Lectures on algebraic topology}. Springer-Verlag,
  Berlin-Heidelberg-New York, 1972.

  
\bibitem{GBF10}
Greg Friedman, \emph{Singular chain intersection homology for traditional and
  super-perversities}, Trans. Amer. Math. Soc. \textbf{359} (2007), 1977--2019.

\bibitem{GBF20}
\bysame, \emph{Intersection homology {K}\"unneth theorems}, Math. Ann.
  \textbf{343} (2009), no.~2, 371--395.

\bibitem{GBF18}
\bysame, \emph{On the chain-level intersection pairing for {PL}
  pseudomanifolds}, Homology, Homotopy and Applications \textbf{11} (2009),
  261--314.

\bibitem{GBF23}
\bysame, \emph{Intersection homology with general perversities}, Geometriae
  Dedicata \textbf{148} (2010), 103--135.
  
\bibitem{GBF26}
\bysame, \emph{An introduction to intersection homology with general
  perversity functions}, Topology of Stratified Spaces, Mathematical Sciences
  Research Institute Publications, vol.~58, Cambridge University Press, 2011,
  pp.~177--222.


\bibitem{GBF25}
Greg Friedman and James McClure, \emph{Cup and cap products in intersection (co)homology}, see arxiv.org/abs/1106.4799.

\bibitem{GM1}
Mark Goresky and Robert MacPherson, \emph{Intersection homology theory},
  Topology \textbf{19} (1980), 135--162.

\bibitem{GM2}
\bysame, \emph{Intersection homology {II}}, Invent. Math. \textbf{72} (1983),
  77--129.

\bibitem{GS83}
Mark Goresky and Paul Siegel, \emph{Linking pairings on singular spaces},
  Comment. Math. Helvetici \textbf{58} (1983), 96--110.

\bibitem{Ha}
Allen Hatcher, \emph{Algebraic topology}, Cambridge University Press,
  Cambridge, 2002.

\bibitem{HMM} Andrew Hassell, Rafe Mazzeo, Richard Melrose, 
\emph{A signature formula for manifolds with corners of codimension two},  
Topology  36  (1997),  no. 5, 1055--1075.


\bibitem{HHM}
Tam{\'a}s Hausel, Eug{\'e}nie Hunsicker, and Rafe Mazzeo, \emph{Hodge
  cohomology of gravitational instantons}, Duke Math. J. \textbf{122} (2004),
  no.~3, 485--548.
 

\bibitem{Hirz}
Friedrich Hirzebruch, \emph{Topological methods in algebraic geometry},
  Classics in Mathematics, Springer-Verlag, Berlin, 1995, Translated from the
  German and Appendix One by R. L. E. Schwarzenberger, With a preface to the
  third English edition by the author and Schwarzenberger, Appendix Two by A.
  Borel, Reprint of the 1978 edition.
  
  
\bibitem{Hun07} E. Hunsicker, \emph{Hodge and Signature Theorems for a family of manifolds with fibration boundary,} Geometry \& Topology 11 (2007) 1581-1622.


\bibitem{Ki}
Henry~C. King, \emph{Topological invariance of intersection homology without
  sheaves}, Topology Appl. \textbf{20} (1985), 149--160.

\bibitem{KL2} Paul Kirk and Matthias Lesch, \emph{The $\eta$-invariant, Maslov index, and spectral flow for Dirac-type operators on manifolds with boundary}.  Forum Math.  16  (2004),  no. 4, 553--629. 

\bibitem{KirWoo}
Frances Kirwan and Jonathan Woolf, \emph{An introduction to intersection
  homology theory. second edition}, Chapman \& Hall/CRC, Boca Raton, FL, 2006.


\bibitem{Lev07}
Filipp Levikov, \emph{Wang sequences in intersection homology}, Diploma thesis,
  Universit{\"a}t Heidelberg, 2007.

\bibitem{Lev10}
\bysame, \emph{Intersection homology wang sequence}, Topology of Stratified Spaces, Mathematical Sciences
  Research Institute Publications, vol.~58, Cambridge University Press, 2011,
  pp.~251--279.




\bibitem{MV86}
Robert MacPherson and Kari Vilonen, \emph{Elementary construction of perverse
  sheaves}, Invent. Math. \textbf{84} (1986), 403--435.

\bibitem{McC}
J.E. McClure, \emph{On the chain-level intersection pairing for {PL}
  manifolds}, Geom. Topol. \textbf{10} (2006), 1391--1424.

\bibitem{Mc78}
Clint Mc{C}rory, \emph{Stratified general position}, Algebraic and geometric
  topology (Proc. Sympos., Univ. California, Santa Barbara, Calif. 1977)
  (Berlin), Lecture Notes in Math., vol. 664, Springer, 1978, pp.~142--146.

\bibitem{MH}
J.~Milnor and D.~Husemoller, {\bf Symmetric bilinear forms}, Springer Verlag,
  New York, 1973.
  
\bibitem{Ro}  Vladimir Rokhlin, \emph{New results in the theory of four-dimensional manifolds},
   Doklady Acad. Nauk. SSSR (N.S.) 84 (1952) 221-224.
   
\bibitem{SS} L. Saper and M. Stern, \emph{$L^2$ cohomology of arithmetic varieties,} Ann. of Math
(2) {\bf 132} (1990), 1-69.

\bibitem{Sa05}
Martintxo Saralegi-Aranguren, \emph{de {R}ham intersection cohomology for
  general perversities}, Illinois J. Math. \textbf{49} (2005), no.~3, 737--758
  (electronic).

\bibitem{Si83}
P.H. Siegel, \emph{Witt spaces: a geometric cycle theory for {KO}-homology at
  odd primes}, American J. Math. \textbf{110} (1934), 571--92.
  
\bibitem{Thom} Rene Thom, \emph{Vari\'et\'es diff\'erentiables cobordantes},
C. R. Acad. Sci. Paris 236, (1953). 1733--1735. 

\bibitem{Va} Boris Vaillant, \emph{Index and Spectral Theory for Manifolds with Generalized 
Fibred Cusps}, arXiv:math/0102072.

\bibitem{Wa69}
C.~T.~C. Wall, \emph{Non-additivity of the signature}, Invent. Math. \textbf{7}
  (1969), 269--274.

\end{thebibliography}
\end{document}